\documentclass[a4paper]{article}
\usepackage{amsmath,amsthm,amssymb,stmaryrd}
\usepackage{url}
\usepackage{todonotes}

\newtheorem{theorem}{Theorem}[section]
\newtheorem{proposition}[theorem]{Proposition}
\newtheorem{lemma}[theorem]{Lemma}
\newtheorem{corollary}[theorem]{Corollary}
\newtheorem{remark}[theorem]{Remark}
\newtheorem{conjecture}[theorem]{Conjecture}
\theoremstyle{definition}

\newtheorem{definition}[theorem]{Definition}

\newcommand{\dom}{\operatorname{dom}}
\newcommand{\powset}{\mathcal{P}}
\newcommand{\nat}{\mathbb{N}}
\newcommand{\baire}{\nat^\nat}
\newcommand{\natinfty}{\nat_\infty}
\newcommand{\klone}{\mathcal{K}_1}
\newcommand{\kltwo}{\mathcal{K}_2}
\newcommand{\vs}{V^{(\mathcal{S})}}
\newcommand{\vl}{V^{(\mathcal{L}_2)}}
\newcommand{\vkone}{V(\klone)}
\newcommand{\vktwo}{V^{\mathbb{P}}}
\newcommand{\vtruth}{V^\ast}
\newcommand{\Vdashtr}{\Vdash_{tr}}
\newcommand{\topln}{\mathcal{L}_n}
\newcommand{\vln}{V^{(\topln)}}
\newcommand{\pair}{\mathbf{p}}

\newcommand{\axiomf}{\mathbf}
\newcommand{\theoryf}{\mathbf}

\newcommand{\ip}[2]{\axiomf{IP}_{{#1},{#2}}}
\newcommand{\ipnn}{\ip{\mathcal{F}_n}{\nat^\nat}}
\newcommand{\llpo}{\axiomf{LLPO}}
\newcommand{\lpo}{\axiomf{LPO}}
\newcommand{\markov}{\axiomf{MP}}
\newcommand{\church}{\axiomf{CT}_0}
\newcommand{\churchu}{\axiomf{CT}_!}
\newcommand{\choice}{\axiomf{AC}}
\newcommand{\rea}{\axiomf{REA}}
\newcommand{\acn}[2]{\choice_{\nat,#1}^{#2}}
\newcommand{\acx}[2]{\choice_{X,#1}^{#2}}
\newcommand{\acb}[2]{\choice_{\baire,#1}^{#2}}
\newcommand{\acndn}[1]{\acn{#1}{\neg \neg}}
\newcommand{\acxdn}[1]{\acx{#1}{\neg \neg}}
\newcommand{\hchoice}{\axiomf{HAC}}
\newcommand{\hacn}{\hchoice_{\nat,\nat}}
\newcommand{\hacb}{\hchoice_{\baire,\nat}}
\newcommand{\hacx}{\hchoice_{X,\nat}}
\newcommand{\fan}{\axiomf{Fan}}
\newcommand{\bim}{\axiomf{BI}_M}

\newcommand{\contbn}{\axiomf{Cont}(\baire, \nat)}
\newcommand{\lcp}{\axiomf{LCP}}
\newcommand{\rdc}{\axiomf{RDC}}
\newcommand{\contac}{\axiomf{CC}}

\newcommand{\czf}{\theoryf{CZF}}
\newcommand{\izf}{\theoryf{IZF}}

\newcommand{\exlte}[1]{\exists^{\leq #1}}

\newcommand{\prf}{\operatorname{Pr}}
\newcommand{\godelno}[1]{\ulcorner #1 \urcorner}

\newcommand{\nil}{\mathtt{nil}}
\newcommand{\maketree}{\operatorname{Tr}}
\newcommand{\cov}{\operatorname{Cover}}

\newcommand{\shp}{\operatorname{S}}
\newcommand{\dt}{\operatorname{D}}

\title{Lifschitz Realizability as a Topological Construction}
\author{Michael Rathjen and Andrew W Swan}

\begin{document}

\maketitle
\begin{abstract}
  We develop a number of variants of Lifschitz realizability for
  $\czf$ by building topological models internally in certain
  realizability models. We use this to show some interesting
  metamathematical results about constructive set theory with variants
  of $\llpo$ including consistency with unique Church's thesis,
  consistency with some Brouwerian principles and variants of the
  numerical existence property.
\end{abstract}

\section{Introduction}

In \cite{vanoostentworemarks} and \cite{vanoosten}, Van Oosten shows
how the Lifschitz realizability topos can be viewed as a category of
sheaves over a particular Lawvere-Tierney topology constructed in the
effective topos. Although a remarkable result, it has some
shortcomings:
\begin{enumerate}
\item The construction refers explicitly to computable functions and
  Lifschitz's encoding of finite sets. This makes it appear that the
  construction is unique to the effective topos and cannot be carried
  out in other toposes.
\item The construction relies on many technical definitions and
  techniques from topos theory.
\item The construction is not guaranteed to work predicatively.
\end{enumerate}

In this paper we will give a new presentation of this result. Instead
of topos theory we work in the set theory $\czf$, which is regarded as
a predicative theory for mathematics. Instead of Lawvere-Tierney
topologies, we will use formal topologies and a predicative notion of
topological model due to Gambino.

Aside from this difference in presentation, our results are more
general than Van Oosten's in two ways (although the first of these
does relate to some more recent results by Lee and Van Oosten in
\cite{leevanoosten}).

Firstly, instead of considering just one formal topology, we will
consider an infinite family of formal topologies $\mathcal{L}_n$ for
each natural number $n \geq 2$, with the original Lifschitz
realizability model just corresponding to the formal topology
$\mathcal{L}_2$. The topologies $\mathcal{L}_n$ correspond to certain
variants of $\llpo$, which were first studied by Richman in
\cite{richmanllpon}, and are denoted $\llpo_n$. We will use these
models to give a new proof of a theorem due to Hendtlass and Lubarsky
in \cite{hendtlasslubarsky}: $\llpo_{n + 1}$ is strictly weaker than
$\llpo_n$. This answers positively a question raised by Hendtlass: is
there a variant of Lifschitz realizability that separates $\llpo_{n}$
from $\llpo_{n + 1}$?

Secondly, we identify axioms, $\ipnn$ that hold in the McCarty
realizability model $V(\klone)$ that suffice to carry out internally
the construction of the formal topologies $\mathcal{L}_n$ we will use
in the models. This can be done entirely in $\czf + \markov + \ipnn$,
without any explicit reference to computable functions. This enables
us to easily generate variants of Lifschitz realizability by simply
checking that the same axioms $\ipnn$ hold in other realizability
models. By using realizability with truth in this way we will show
that the theories $\czf + \markov + \llpo_n$ have certain variants of
the numerical existence property. By using realizability over $\kltwo$
in this way we will show that $\czf + \llpo$ is consistent with
certain (but not all) Brouwerian continuity principles.

A more traditional version of Lifschitz realizability for $\czf +
\llpo + \churchu$ similar to that in \cite{rathjenchen} can be
recovered by a two step process of interpreting the topological model
$\vl$ in the McCarty realizability model $V(\klone)$, itself
constructed in $\czf + \markov$ as illustrated below.
\[
\begin{array}{l|ccccc}
  \text{Theory} & \czf + \markov
  & & \czf + \markov
  & &  \czf + \markov \\
  & + \llpo + \churchu & & + \ip{\mathcal{F}_2}{\baire} + \church & \\
  & & \hookrightarrow & & \hookrightarrow & \\
  \text{Model} & \vl & & V(\klone) & & V
\end{array}
\]

\nocite{mylatzthesis}
\nocite{richmanllpon}

\section{Constructive Set Theory}

We will consider the intuitionistic set theories $\czf$ and $\izf$, as
described for instance in \cite{aczelrathjen} or
\cite{aczelrathjenbookdraft}.

We will use the following set theoretic formulations of Markov's
principle and Church's thesis.

\begin{definition}
  \label{def:8}
  \emph{Markov's principle}, $\markov$, is the following axiom. Let
  $\alpha : \nat \rightarrow 2$ be a function. Then,
  \begin{equation*}
    \label{eq:59}
    \neg \neg \; (\exists n \in \nat)\,\alpha(n) = 1 \quad \rightarrow
    \quad (\exists n \in \nat)\,\alpha(n) = 1
  \end{equation*}
\end{definition}

\begin{definition}
  \label{def:9}
  \emph{Church's thesis}, $\church$ is the following axiom. Let
  $\phi(x, y)$ be any formula. Then, writing $\{e\}(n)$ to mean the
  result of running the $e$th Turing machine with input $n$,
  \begin{equation*}
    \label{eq:63}
    (\forall n \in \nat)(\exists m \in \nat)\,\phi(n, m) \quad
    \rightarrow \quad
    (\exists e \in \nat)(\forall n \in \nat)\,\phi(n, \{e\}(n))
  \end{equation*}

  \emph{Church's thesis for functions}, $\churchu$ is the axiom that
  every function from $\nat$ to $\nat$ is computable.
\end{definition}

We recall the following definitions and theorems on finite
sets, as appear in \cite[Chapters 6 and
8]{aczelrathjenbookdraft}. The theorems will often be used
implicitly while working with finitely enumerable sets.

\begin{definition}
  A set $X$ is \emph{finite} if for some $n \in \nat$ there
  exists a bijection from $n$ to $X$.
  
  A set $X$ is \emph{finitely enumerable} if for some $n \in \nat$
    there exists a surjection from $n$ to $X$.
\end{definition}

\begin{theorem}[$\czf$]
  Suppose that $\phi(x_1,\ldots,x_n)$ is a formula of arithmetic,
  where all quantifiers are bounded, and the only free variables are
  amongst $x_1,\ldots,x_n$. Then we can prove the following instance
  of excluded middle.
  \begin{equation*}
    (\forall x_1,\ldots,x_n \in \nat)\;\phi(x_1,\ldots,x_n) \vee \neg
    \phi(x_1,\ldots,x_n)
  \end{equation*}
\end{theorem}

\begin{proof}
  See \cite[Theorem 6.6.2]{aczelrathjenbookdraft}.
\end{proof}

\begin{theorem}[$\czf$] ``The Pigeonhole Principle for Finitely
  Enumerable Sets.'' Let $A$ be a finitely enumerable set. Every
  injective function $f : A \rightarrowtail A$ is also a surjection.
\end{theorem}

\begin{proof}
  See \cite[Theorem 8.2.10]{aczelrathjenbookdraft}.
\end{proof}

\begin{theorem}[$\czf$] ``The Finite Axiom of Choice.''
  Suppose $A$ is a finite set, $B$ is any set, and $R \subseteq A
  \times B$ is a relation such that $(\forall a \in A)(\exists b \in
  B)\, \langle a, b \rangle \in R$.

  Then there is a function $f : A \rightarrow B$ such that for all $a
  \in A$, $\langle a, f(a) \rangle \in R$.
\end{theorem}

\begin{proof}
   See \cite[Theorem 8.2.8]{aczelrathjenbookdraft}.
\end{proof}

We can also prove a finite version of $\lpo$:
\begin{theorem}[$\czf$]
  For every finitely enumerable set $X$ and every
  $f : X \rightarrow 2$, either there exists some $x \in X$ such that
  $f(x) = 1$ or for all $x \in X$, $f(x) = 0$.
\end{theorem}

\begin{proof}
  Show by induction on $n$ that if there is a surjection $n
  \twoheadrightarrow X$ then the result holds for $X$.
\end{proof}

\section{Formal Topologies and Heyting Valued Models of $\czf$}

\subsection{Basic Definitions}

We recall the basic definitions of formal topology and Gambino's
Heyting valued interpretation of $\czf$. For details see
\cite{gambinohvi}. The basic idea here is that to each formula in set
theory, we assign an open set, which we think of as the ``truth
value'' of the formula. We use Gambino's presentation of topological
models since it can be formalised in, and provides models for $\czf$.

\begin{definition}
  If $\langle S, \leq \rangle$ is a poset, and $p$ is a subset of $S$,
  we write $p \downarrow$ for the downwards closure of $p$. That is,
  \begin{equation*}
    p \downarrow \quad:=\quad \{ x \in S \;|\; (\exists y \in p)\,x \leq y \}
  \end{equation*}
\end{definition}

\begin{definition}
  A \emph{formal topology} is $\langle S, \leq, \triangleleft \rangle$
  such that $\langle S, \leq \rangle$ is a poset, and $\triangleleft$
  is a (class) relation between elements and subsets of $S$, such that
  \begin{enumerate}
  \item \label{ftreflexivity} if $a \in p$, then $a \triangleleft p$
  \item \label{ftdownclosed} if $a \leq b$ and $b \triangleleft p$,
    then $a \triangleleft p$
  \item \label{fttransitive} if $a \triangleleft p$ and $(\forall x
    \in p)(x \triangleleft q)$, then $a \triangleleft q$
  \item \label{ftintersect} if $a \triangleleft p$ and $a
    \triangleleft q$, then $a \triangleleft \downarrow p \cap
    \downarrow q$
  \end{enumerate}
\end{definition}

\begin{definition}
  Let $\mathcal{S} := \langle S, \leq, \triangleleft \rangle$ be a
  formal topology. A \emph{set-presentation} for $\mathcal{S}$ is a
  (set) function $R : S \rightarrow \powset (\powset S)$ such that
  \begin{equation*}
    \label{eq:1}
    a \triangleleft p \leftrightarrow (\exists u \in R(a))u \subseteq p
  \end{equation*}
  If $(S, \leq, \triangleleft)$ has a set-presentation, we say it is
  \emph{set-presentable}.
\end{definition}

\begin{definition}
  \label{def:1}
  Let $\mathcal{S} := \langle S, \leq, \triangleleft \rangle$ be a set
  presentable formal topology. We define the \emph{nucleus} of
  $\mathcal{S}$ to be the following class function $j : \powset(S)
  \rightarrow \powset(S)$. For $p \subseteq S$,
  \begin{equation*}
    \label{eq:8}
    j(p) := \{ a \in S \;|\; a \triangleleft p \}
  \end{equation*}

  We extend $j$ to an operation, $J$, on subclasses of $S$ by
  \begin{equation*}
    \label{eq:50}
    J(P) := \bigcup \{ j(v) \;|\; v \subseteq P \}
  \end{equation*}
\end{definition}

\begin{definition}
  \label{def:6}
  We say a formal topology $\langle S, \leq, \triangleleft \rangle$ is
  \emph{proper} if for all $a \in S$, $\neg a \triangleleft
  \emptyset$. (Or equivalently if $j(\emptyset) = \emptyset$.)
\end{definition}

\begin{definition}
  \label{def:5}
  Let $\mathcal{S} = \langle S, \leq, \triangleleft \rangle$ be a set
  presentable formal topology. The class $\vs$ is defined inductively
  as the smallest class such that $f \in \vs$ whenever $f$ is a
  function with $\dom(f) \subseteq \vs$ and for all $x \in \dom(f)$,
  $f(x)$ is a $\triangleleft$-closed subset of $S$.
\end{definition}

For each sentence $\phi$ in the language of set theory with parameters
from $\vs$, we assign a $\triangleleft$-closed class
denoted $\llbracket \phi \rrbracket$, which we define by induction
on formulas as follows. For bounded $\phi$, $\llbracket \phi
\rrbracket$ will be a set.

We first define a complete Heyting algebra structure on the class of
$\triangleleft$-closed classes as follows. For $P$ and $Q$
$\triangleleft$-closed classes,
\begin{align*}
  \top &:= S \\
  \bot &:= J(\emptyset) \\
  P \wedge Q &:= P \cap Q \\
  P \vee Q &:= J(P \cup Q) \\
  P \rightarrow Q &:= \{ a \in S \;|\; a \in P \rightarrow a \in Q \}
  \\
  \bigvee_{x \in U} P_x &:= J\left( \bigcup_{x \in U} P_x \right) \\
  \bigwedge_{x \in U} P_x &:= \bigcap_{x \in U} P_x
\end{align*}

We define the interpretation of atomic sentences $a \in b$ and $a = b$
by simultaneous induction on $a$ and $b$:
\begin{align*}
  a \in b &:= \bigvee_{c \in \dom(b)} b(c) \wedge \llbracket a = c \rrbracket \\
  a = b &:= \bigwedge_{c \in \dom(a)} a(c) \rightarrow \llbracket c
  \in b \rrbracket \quad \wedge \quad \bigwedge_{c \in \dom(b)} b(c)
  \rightarrow \llbracket c \in b \rrbracket
\end{align*}

We then extend this to all formulas as below.
\begin{align*}
  \llbracket \bot \rrbracket &:= \bot \\
  \llbracket \phi \wedge \psi \rrbracket &:= \llbracket \phi
  \rrbracket \wedge \llbracket \psi \rrbracket \\
  \llbracket \phi \vee \psi \rrbracket &:= \llbracket \phi
  \rrbracket \vee \llbracket \psi \rrbracket \\
  \llbracket \phi \rightarrow \psi \rrbracket &:= \llbracket \phi
  \rrbracket \rightarrow \llbracket \psi \rrbracket \\
  \llbracket (\exists x \in a)\,\phi \rrbracket &:= \bigvee_{x \in
    \dom(a)} \llbracket \phi \rrbracket \\
  \llbracket (\forall x \in a)\,\phi \rrbracket &:= \bigwedge_{x \in
    \dom(a)} \llbracket \phi \rrbracket \\
  \llbracket (\exists x)\,\phi \rrbracket &:= \bigvee_{x \in
    \vs} \llbracket \phi \rrbracket \\
  \llbracket (\forall x)\,\phi \rrbracket &:= \bigwedge_{x \in
    \vs} \llbracket \phi \rrbracket
\end{align*}

We write $\vs \models \phi$ to mean $\llbracket \phi
\rrbracket = \top$. For a collection of formulas, $\Phi$, we write
$\vs \models \Phi$ to mean $\vs \models
\phi$ for all $\phi \in \Phi$.

\begin{theorem}[Gambino]
  \label{thm:1}
  Let $\mathcal{S}$ be a set presentable formal topology. Then
  \begin{equation*}
    \label{eq:51}
    \vs \models \czf
  \end{equation*}
\end{theorem}

\begin{proof}
  See \cite{gambinohvi}.
\end{proof}

\subsection{Some Absoluteness Lemmas}
\label{sec:some-absol-lemm}

For some of the results later, it will be important that under certain
conditions statements that hold in the background universe also hold
internally in the topological model and vice versa. To this end, we
prove a series of absoluteness lemmas below.

First note that any set $x$ can be viewed as an element of $\vs$,
$\hat{x}$ as follows.
\begin{align*}
  \dom(\hat{x}) &:= x \\
  \hat{x}(y) &:= \top &\text{for all } y \in x
\end{align*}

\begin{lemma}
  \label{lem:abslem}
  In the below, let $\phi$ and $\psi$ be any formulas, possibly with
  parameters from $\vs$.
  \begin{enumerate}
  \item We can prove in $\czf$ that for any set $x$, $\llbracket
    \phi(\hat y) \rrbracket = \top$ holds for all $y$ in $x$ if and
    only if $\llbracket (\forall y \in \hat x)\, \phi(y) \rrbracket =
    \top$ holds.
  \item $\llbracket \phi \rrbracket \subseteq \llbracket \psi
    \rrbracket$ if and only if $\llbracket \phi \rightarrow \psi
    \rrbracket = \top$.
  \item $\llbracket \phi \rrbracket = \top$ and $\llbracket \psi
    \rrbracket = \top$ if and only if $\llbracket \phi \wedge \psi
    \rrbracket = \top$.
  \item For proper formal topologies, $\llbracket \bot
    \rrbracket = \emptyset$.
  \item \label{abslemex} If $(\exists y \in x)\,\llbracket \phi(\hat
    y) \rrbracket = \top$ then $\llbracket (\exists y \in \hat x)\,\phi(y)
    \rrbracket = \top$.
  \item \label{abslemdisj} If $\llbracket \phi \rrbracket = \top$ or
    $\llbracket \psi \rrbracket = \top$ then $\llbracket \phi \vee \psi
    \rrbracket = \top$.
  \end{enumerate}
\end{lemma}

\begin{proof}
  For 1, 2 and 3 note that joins and implications in the Heyting
  algebra on $\triangleleft$-closed classes are exactly the usual joins
  and implications for the Heyting algebra of subsets of a set. 1,
  2 and 3 follow by the basic properties of Heyting algebras.

  4 is just by unfolding definitions.

  For 5, note that we have
  \begin{equation*}
    \label{eq:60}
    \llbracket (\exists y \in \hat x)\,\phi(y)
    \rrbracket = J \left( \bigcup_{y \in x} \, \llbracket
      \phi(\hat{y}) \rrbracket \right)
  \end{equation*}
  However, we also have
  \begin{equation*}
    \label{eq:61}
    \bigcup_{y \in x} \, \llbracket \phi(\hat{y})
    \rrbracket \subseteq J \left( \bigcup_{y \in x} \, \llbracket
      \phi(\hat{y}) \rrbracket \right)
  \end{equation*}
  Then 5 easily follows.

  One can then prove 6 by a similar argument.
\end{proof}

\begin{lemma}
  \label{lem:disjexcriteria}
  Suppose that $(\bigcup_x\,\llbracket \phi(\hat x) \rrbracket) \;
  \subseteq \; \llbracket \psi \rrbracket$. Then $\llbracket ((\exists
  x)\,\phi(x)) \;\rightarrow\; \psi \rrbracket = \top$. Suppose that
  $\llbracket \phi \rrbracket
  \vee \llbracket \psi \rrbracket \subseteq \llbracket \chi
  \rrbracket$. Then $\llbracket \phi \vee \psi \rightarrow \chi
  \rrbracket = \top$.
\end{lemma}

\begin{proof}
  Suppose that $(\bigcup_x\,\llbracket \phi(\hat x) \rrbracket) \;
  \subseteq \; \llbracket \psi \rrbracket$. Then we have
  \begin{equation*}
    \label{eq:62}
    J\left (\bigcup_x\,\llbracket \phi(\hat x) \rrbracket \right) \;
    \subseteq \; J \left( \llbracket \psi \rrbracket \right)
  \end{equation*}
  However, $\llbracket \psi \rrbracket$ is already
  $\triangleleft$-closed, so $J ( \llbracket \psi \rrbracket ) =
  \llbracket \psi \rrbracket$. But then it easily follows that
  $\llbracket (\exists x)\,\phi(x) \rrbracket \subseteq \llbracket
  \psi \rrbracket$ and so $\llbracket ((\exists x)\,\phi(x))
  \;\rightarrow\; \psi \rrbracket = \top$.

  The other part can be proved by a similar argument.
\end{proof}

\begin{lemma}
  \label{lem:pairunabs}
  Let $x$ and $y$ be sets and let $z \in \vs$. Then,
  \begin{align}
    \llbracket z \in \widehat{\{x, y\}} \;\leftrightarrow\; z =
    \hat{x} \vee z = \hat{y} \llbracket &= \top \label{eq:pairabs}\\
    \llbracket z \in \widehat{\bigcup x} \; \leftrightarrow \;
    (\exists w \in \hat{x})\,z \in w \rrbracket &= \top \label{eq:unionabs}
  \end{align}
\end{lemma}

\begin{proof}
  We first check \eqref{eq:pairabs}. Unfolding definitions we have
  that both $\llbracket z \in \widehat{\{x, y\}} \rrbracket$ and
  $\llbracket z = \hat{x} \vee z = \hat{y} \rrbracket$ are equal to
  $j(\llbracket z = \hat{x} \rrbracket \cup \llbracket z = \hat{y}
  \rrbracket)$. It easily follows that \eqref{eq:pairabs} holds.

  We now check \eqref{eq:unionabs}. Unfolding definitions we have the
  following.
  \begin{align*}
    \llbracket z \in \widehat{\bigcup x} \rrbracket &= j(\bigcup_{v
      \in x} \bigcup_{w \in v} \llbracket z = \hat{w} \rrbracket) \\
    \llbracket (\exists w \in \hat{x})\, z \in w \rrbracket &=
    j(\bigcup_{v \in x} j(\bigcup_{w \in v} \llbracket z = \hat{w}
    \rrbracket))
  \end{align*}
  By monotonicity of $j$ and union we have $\llbracket z \in
  \widehat{\bigcup x} \rrbracket \subseteq \llbracket (\exists w \in
  \hat{x})\, z \in w \rrbracket$. We now check $\llbracket z \in
  \widehat{\bigcup x} \rrbracket \supseteq \llbracket (\exists w \in
  \hat{x})\, z \in w \rrbracket$. By
  axiom \ref{fttransitive} of the definition of formal topology, it
  suffices to check that $\bigcup_{v \in x} j(\bigcup_{w \in
    v} \llbracket z = \hat{w} \rrbracket) \subseteq
  j(\bigcup_{v \in x} \bigcup_{w \in v} \llbracket z = \hat{w}
  \rrbracket)$. Let $a \in \bigcup_{v \in x} j(\bigcup_{w \in v}
  \llbracket z = \hat{w} \rrbracket)$. Then for some $v \in x$, we
  have $a \in j(\bigcup_{w \in v} \llbracket z = \hat{w}
  \rrbracket)$. But now $a \in j(\bigcup_{v' \in x} \bigcup_{w \in v'}
  \llbracket z = \hat{w} \rrbracket$ by monotonicity of $j$, as
  required.
\end{proof}

\begin{lemma}
  The natural numbers are absolute, in the following sense.
  \begin{equation*}
    \llbracket (\forall u) \; [u \in \hat \nat \,\leftrightarrow\,
    (\emptyset = u \vee (\exists v \in \hat \nat)\,u = v \cup \{v\})]
    \rrbracket = \top
  \end{equation*}
\end{lemma}

\begin{proof}
  First note that $\llbracket (\forall u \in \hat{\nat})\, u \cap
  \{u\} \in \hat{\nat} \rrbracket = \bigcap_{u \in \nat} \llbracket
  \hat u \cup \{\hat u\} \in \hat{\nat} \rrbracket$ but this is equal
  to $\top$ by lemma \ref{lem:pairunabs} and the fact that $\{u\} \cup
  u \in \nat$ for every $u \in \nat$. We also easily have $\llbracket
  \emptyset \in \hat{\nat} \rrbracket$. But we have now shown one half
  of the bi-implication:
  \begin{equation*}
    \label{eq:37}
    \llbracket (\forall u) \; [u \in \hat \nat \,\rightarrow\,
    (\emptyset = u \vee (\exists v \in \hat \nat)\,u = v \cup \{v\})]
    \rrbracket = \top
  \end{equation*}

  Now assume that for some $v \in \nat$, $a \in \llbracket u = \hat{v}
  \cup \{\hat{v}\} \rrbracket$. Then using the soundness of the laws
  of equality, we have $\llbracket u = \hat{v} \cup \{\hat{v}\}
  \rrbracket \cap \llbracket \hat{v} \cup \{\hat{v}\} \in \hat{\nat}
  \rrbracket \subseteq \llbracket u \in \hat{\nat} \rrbracket$. Hence
  $a \in \llbracket u \in \hat{\nat} \rrbracket$. But we now apply
  both parts of lemma \ref{lem:disjexcriteria} to deduce
  \begin{equation*}
    \llbracket (\forall u) \; [u \in \hat \nat \,\leftarrow\,
    (\emptyset = u \vee (\exists v \in \hat \nat)\,u = v \cup \{v\})]
    \rrbracket = \top
  \end{equation*}
  which is the other half of the bi-implication we require.
\end{proof}

\begin{lemma}
  Suppose that $\langle S, \leq, \triangleleft\rangle$ is a proper
  formal topology. Then equality and membership are absolute for the
  natural numbers in the following sense. For every $m, n \in \nat$,
  we have that either $\llbracket \hat m = \hat n \rrbracket = \top$
  or $\llbracket \hat m = \hat n \rrbracket = \emptyset$, $m = n$ if
  and only if $\llbracket \hat m = \hat n \rrbracket = \top$, either
  $\llbracket \hat m \in \hat n \rrbracket = \top$ or $\llbracket \hat
  m \in \hat n \rrbracket = \emptyset$ and $m \in n$ if and only if
  $\llbracket \hat m \in \hat n \rrbracket = \top$.
\end{lemma}

\begin{proof}
  These are proved simultaneously by induction on $n$ and $m$.
\end{proof}

\begin{lemma}
  Finite tuples are absolute, in the following sense. We can show in
  $\czf$ that for every set $x$ and every $n \in \nat$ and every set
  $z$,
  \begin{equation*}
    \label{eq:34}
    \llbracket z \in \widehat{x^n} \;\leftrightarrow\; z \in
    \hat{x}^{\hat{n}} \rrbracket = \top
  \end{equation*}
\end{lemma}

\begin{proof}
  This can be proved by induction on $n$.
\end{proof}

\begin{lemma}
  Let $x$ be a set. Then function application for $\nat^x$ is
  absolute, in the sense that for $f \in \nat^x$, $z \in x$ and
  $n \in \nat$, $f(z) = n$ if and only if
  $\llbracket \hat{f}(\hat{z}) = \hat{n} \rrbracket = \top$.
\end{lemma}

\begin{proof}
  Note that the formula $\hat{f}(\hat{z}) = \hat{n}$ is equivalent to
  the following
  \begin{equation*}
    \label{eq:38}
    (\forall w \in \hat{f})(\forall v \in \hat{x})(\forall u \in
    \hat{\nat})\, w = \langle v, u \rangle \rightarrow u = \hat{n}
  \end{equation*}
  This is clearly absolute by the previous lemmas.
\end{proof}

\begin{remark}
  In \cite{gambinohvi} it is stated that all restricted formulas are
  absolute. This is not provable in $\izf$ or $\czf$, since the
  converses to parts \ref{abslemex} and \ref{abslemdisj} of lemma
  \ref{lem:abslem} do not hold in general and atomic formulas are not
  in general absolute. The double negation formal
  topology provides a counterexample, as do the formal topologies
  $\mathcal{L}_n$ considered in this paper. Also note that properness
  is necessary to show that $\bot$ is absolute.
\end{remark}

\section{$\llpo$ and $\llpo_n$}

\subsection{An Alternative Formulation of $\llpo$}

We will first show how $\llpo$ can be formulated in terms of the poset
$\natinfty$ defined below. This formulation will motivate the
definition of the formal topology as the simplest one making $\llpo$
true in the topological model (based on an observation of Van Oosten
in \cite{vanoostentworemarks}).

\begin{definition}
  \label{def:2}
  Let $\natinfty$ be the set of decreasing binary sequences, i.e.
  \begin{equation*}
    \label{eq:3}
    \natinfty := \{ \alpha : \nat \rightarrow 2 \;|\; (\forall i \leq
    j)\,\alpha(j) \leq \alpha(i) \}
  \end{equation*}
\end{definition}
We will consider $\natinfty$ as a poset with the pointwise ordering,
i.e. $\alpha \leq \beta$ if for all $i \in \nat$, $\alpha(i) \leq \beta(i)$.

\begin{proposition}
  \label{prop:1}
  If $\alpha, \beta \in \natinfty$, then the join $\alpha \vee \beta$
  exists and is defined pointwise, i.e. for $i \in \nat$
  \begin{equation*}
    \label{eq:4}
    (\alpha \vee \beta)(i) := \alpha(i) \vee \beta(i)
  \end{equation*}

  Hence, if $F$ is a finitely enumerable subset of $\natinfty$, then
  $\bigvee F$ exists and is defined pointwise.
\end{proposition}

The top element of $\natinfty$ is the function constantly equal to
$1$. We'll write this function as $1$.

\begin{lemma}
  \label{lem:onestable}
  For all $\alpha \in \natinfty$, we have $\neg \neg \alpha = 1
  \rightarrow \alpha = 1$.
\end{lemma}

\begin{proof}
  Suppose $\neg \neg \alpha = 1$. For each $i \in \nat$, we have that
  $\alpha(i)$ is either $0$ or $1$. But if $\alpha(i) = 0$, then we
  would have $\neg \alpha = 1$, contradicting $\neg \neg \alpha =
  1$. Hence $\alpha(i) = 1$ for all $i \in \nat$, and so $\alpha = 1$.
\end{proof}

\begin{lemma}
  \label{lem:meetnotone}
  Assume Markov's principle. Suppose that $\mathcal{F} \subseteq
  \natinfty$ is a finitely enumerable set such that $\bigwedge
  \mathcal{F} \neq 1$. Then for some $\alpha \in \mathcal{F}$, $\alpha
  \neq 1$.
\end{lemma}

\begin{proof}
  Suppose $\bigwedge \mathcal{F} \neq 1$. Then by Markov's principle,
  there is some $n$ such that $\bigwedge \mathcal{F} (n) =
  0$. However, we now clearly have $\alpha(n) = 0$ for some $\alpha
  \in \mathcal{F}$ (since $\{ \alpha(n) \;|\; \alpha \in \mathcal{F}
  \}$ is a finitely enumerable set of natural numbers), and hence
  $\alpha \neq 1$.
\end{proof}

\begin{lemma}
  \label{lem:joinnotone}
  Assume Markov's principle. Suppose that $\mathcal{F} \subseteq
  \natinfty$ is a finitely enumerable set such that for each $\alpha
  \in \mathcal{F}$, $\alpha \neq 1$. Then $\bigvee \mathcal{F} \neq
  1$.
\end{lemma}

\begin{proof}
  Since $\mathcal{F}$ is finitely enumerable, we can write
  $\mathcal{F} = \{\alpha_1,\ldots,\alpha_k\}$. By Markov's principle
  we have for each $i$, $n_i$ such that $\alpha_i(n_i) = 0$. Take $N
  := \max_i n_i$. Then we have that $(\bigvee \mathcal{F}) (N) = 0$
  and therefore $\bigvee \mathcal{F} \neq 1$.
\end{proof}

Recall that $\llpo$ is usually formulated as below.

\begin{definition}
  \label{def:7}
  The lesser limited principle of omniscience ($\llpo$) is the
  following axiom. Let $\alpha : \nat \rightarrow 2$ be a binary
  sequence such that for all $i, j \in \nat$, if $\alpha(i) =
  \alpha(j) = 1$ then $i = j$. Then either for all $i \in \nat$,
  $\alpha(2i) = 0$, or for all $i \in \nat$ $\alpha(2i + 1) = 0$.
\end{definition}

We now obtain the equivalent presentations of $\llpo$ below.

\begin{proposition}
  \label{prop:llpoversions}
  The following are equivalent:
  \begin{enumerate}
  \item $\llpo$
  \item for all $\alpha, \beta \in \natinfty$, if $\alpha \vee \beta =
    1$, then $\alpha = 1$ or $\beta = 1$
  \item for all inhabited finitely enumerable sets $F \subseteq
    \natinfty$, if $\bigvee F = 1$, then there exists $\alpha \in F$
    such that $\alpha= 1$
  \end{enumerate}
\end{proposition}

\begin{proof}
  To show $1 \Rightarrow 2$, let $\alpha, \beta \in \natinfty$ be such
  that $\alpha \vee \beta = 1$. Then define $\gamma : \nat
  \rightarrow 2$ as below.
  \begin{equation*}
    \label{eq:58}
    \gamma(i) =
    \begin{cases}
      1 & \text{if } i = 2j, \alpha(j) = 1 \text{ and }\alpha(j + 1) =
      0 \\
      1 & \text{if } i = 2j + 1, \beta(j) = 1 \text{ and }\beta(j + 1) =
      0 \\
      0 & \text{otherwise}
    \end{cases}
  \end{equation*}
  Then by applying $\llpo$ to $\gamma$, we can show either $\alpha =
  1$ or $\beta = 1$.

  Now to show $2 \Rightarrow 1$, let $\gamma : \nat \rightarrow 2$ be
  such that for all $i, j$ if $\gamma(i) = \gamma(j) = 1$, then $i =
  j$. Define $\alpha$ and $\beta$ as follows.
  \begin{align*}
    \alpha(i) &=
    \begin{cases}
      1 & \text{for all } j \leq i, \gamma(2j) = 0 \\
      0 & \text{for some } j \leq i, \gamma(2j) = 1
    \end{cases} \\
    \beta(i) &=
    \begin{cases}
      1 & \text{for all } j \leq i, \gamma(2j + 1) = 0 \\
      0 & \text{for some } j \leq i, \gamma(2j + 1) = 1
    \end{cases}
  \end{align*}
  Then one can easily check that $\alpha \vee \beta = 1$, and if
  $\alpha = 1$ then $\gamma(2i) = 0$ for all $i$, and if $\beta = 1$
  then $\gamma(2i + 1) = 0$ for all $i$.

  Finally note that $2$ is a special case of $3$, and that $3$ follows
  from $2$ by showing by induction on $n$ that the result holds
  for all $F$ that admit a surjection $n \twoheadrightarrow F$.
\end{proof}

\subsection{Generalising to $\llpo_n$}
\label{sec:llpo_n}

In \cite{richmanllpon}, Richman considered for each $n \geq 2$ a
variant of $\llpo$, that he denoted $\llpo_n$. These axioms were also
studied by Hendtlass and Lubarsky, who showed (amongst other results)
that $\llpo_{n + 1}$ is strictly weaker than $\llpo_n$.  In this
section we show that like $\llpo$, $\llpo_n$ can also be formulated
using $\natinfty$.

\begin{definition}
  Let $n \geq 2$. $\llpo_n$ is the following statement: Let $\alpha :
  \nat \rightarrow 2$ be a binary sequence such that for all $i, j \in
  \nat$, $\alpha(i) = \alpha(j) = 1$ implies $i = j$. Then there is
  some $k$ with $0 \leq k < n$ such that for all $i$, $\alpha(in + k)
  = 0$.
\end{definition}

\begin{remark}
  In \cite{abhkarithhierarchy} Akama, Hayashi, Berardi and Kohlenbach
  studied a separate hierarchy of variants of $\llpo$, denoted
  $\Sigma^0_n-\llpo$. They show (amongst other results) that for each
  $n$, $\Sigma^0_{n + 1}-\llpo$ is strictly stronger than
  $\Sigma^0_n-\llpo$. Another variant of Lifschitz realizability
  (relativised to $\Delta^0_n$ functions) was used for one of their
  separation results.
\end{remark}

We now give the equivalent formulation using $\nat_\infty$.

\begin{proposition}
  Let $n \geq 2$. The following are equivalent:
  \begin{enumerate}
  \item $\llpo_n$
  \item Let $\alpha_1,\ldots,\alpha_n \in \natinfty$ be such that for all $i, j$
    with $1 \leq i \neq j \leq n$, $\alpha_i \vee \alpha_j = 1$. Then there exists
    $i \in \{1,\ldots,n\}$ such that $\alpha_i = 1$.
  \end{enumerate}
\end{proposition}

\begin{proof}
  Similar to the proof of proposition \ref{prop:llpoversions}.
\end{proof}

We now aim towards another characterisation of $\llpo_n$ analogous to
part 3 of proposition \ref{prop:llpoversions} that will be useful
later.

\begin{definition}
  For each $n$, we define the set of \emph{$n$-trees} by the following
  recursive definition.
  \begin{enumerate}
  \item There is an $n$-tree $\nil$.
  \item If we have a list of $n$-trees $T_1,\ldots,T_n$ and a list of
    decreasing sequences $\alpha_1,\ldots,\alpha_n \in \natinfty$,
    then $\maketree(T_1,\ldots,T_n ; \alpha_1,\ldots,\alpha_n)$ is an
    $n$-tree.
  \end{enumerate}
\end{definition}

\begin{definition}
  An $n$-tree is defined to be \emph{good} according to the following
  recursive definition.
  \begin{enumerate}
  \item $\nil$ is good.
  \item $\maketree(T_i ; \alpha_i)$ is good if for any $1 \leq i \neq
    j \leq n$, $\alpha_i \vee \alpha_j = 1$, and for any $1 \leq i \leq
    n$, if $\alpha_i = 1$ then $T_i$ is good.
  \end{enumerate}
\end{definition}

\begin{definition}
  An $n$-tree is defined to be \emph{very good} according to the
  following inductive definition.
  \begin{enumerate}
  \item $\nil$ is very good.
  \item $\maketree(T_i ; \alpha_i)$ is very good if it is good, and
    for some $1 \leq i \leq n$, $\alpha_i = 1$ and $T_i$ is very good.
  \end{enumerate}
\end{definition}

\begin{theorem}
  \label{thm:llpontrees}
  $\llpo_n$ is equivalent to the statement that every good $n$-tree is
  very good.
\end{theorem}

\begin{proof}
  We first assume that every good $n$-tree is very good and deduce
  $\llpo_n$. Let $\alpha_1,\ldots,\alpha_n \in \natinfty$ be such that
  for any $1 \leq i \neq j \leq n$, $\alpha_i \vee \alpha_j = 1$. Then
  note that we can form a good $n$-tree $\maketree(\nil ;
  \alpha_i)$. If $\maketree(\nil ; \alpha_i)$ is very good, then for
  some $i$, $\alpha_i = 1$, as required.

  For the converse, we assume $\llpo_n$ and prove by induction that
  for every $n$-tree, $T$, if $T$ is good then $T$ is very good.

  For $\nil$, this is clear.

  For $T = \maketree(T_i ; \alpha_i)$, assume that $T$ is good. Then
  for $1 \leq i \neq j \leq n$ we have $\alpha_i \vee \alpha_j =
  1$. Hence, for some $i$, $\alpha_i = 1$ by $\llpo_n$. Since $T$ is
  good and $\alpha_i = 1$, we have that $T_i$ is good. But by
  induction we may assume now that $T_i$ is very good. Hence, $T$
  is also very good.
\end{proof}

\begin{definition}
  $\llpo_\infty$ is the following statement. Let $( , ) : \nat \times
  \nat \rightarrow \nat$ be a surjective pairing function, and let
  $\alpha : \nat \rightarrow 2$ be a binary sequence such that
  $\alpha(i) = 1$ for at most one $n$. Then for some $k \in \nat$, and
  for all $n \in \nat$ $\alpha(k, n) = 0$.
\end{definition}

\begin{proposition}
  $\llpo_\infty$ is equivalent to the following statement. Let
  $(\alpha_i)_{i \in \nat}$ be such that $\alpha_i \in \natinfty$ for
  each $i \in \nat$. Suppose further that for $i \neq j$, $\alpha_i
  \vee \alpha_j = 1$. Then for some $i$, $\alpha_i = 1$.
\end{proposition}

\subsection{Absoluteness Results for $n$-Trees}
\label{sec:absol-results-n}

We next show how to encode $n$-trees as functions $\nat \rightarrow
\nat$.

\begin{definition}
  \label{def:shpdata}
  Let $T$ be an $n$-tree. We define the \emph{shape} of $T$, $\shp(T)
  \in \nat$ as follows. Assume that we have a standard way of encoding
  lists of natural numbers as natural numbers such that encoding and
  decoding can be done in a primitive recursive manner and the code
  for a list is greater than each of its elements, and write this
  using brackets $()$.
  \begin{enumerate}
  \item $\shp(\nil)$ is defined to be $()$.
  \item $\shp(\maketree(T_i ; \alpha_i))$ is defined to be
    $(\shp(T_1), \ldots, \shp(T_n))$.
  \end{enumerate}

  We define the data for $T$, $\dt(T) \in 2^\nat$ as follows.
  \begin{enumerate}
  \item $\dt(\nil)(j) := 0$ for all $j \in \nat$.
  \item We define $\dt(\maketree(T_i ; \alpha_i))$ as follows. For any
    $j \in \nat$, $j$ can be written uniquely as either $2 n k + 2 i$
    or $2 n k + 2 i + 1$ where $0 \leq i < n$. We define
    \begin{align*}
      \dt(\maketree(T_i ; \alpha_i))(2 n k + 2 i) &:=
      \alpha_i(k) \\
      \dt(\maketree(T_i ; \alpha_i))(2 n k + 2 i + 1) & :=
      \dt(T_i)(k)
    \end{align*}
  \end{enumerate}
\end{definition}

\begin{lemma}
  \label{lem:goodpr}
  There are primitive recursive functions $b$, $c$, $f$, $g_0$ and $g_1$
  such that an $n$-tree $T$ is good if and only if
  \begin{multline}
    \label{eq:35}
    \forall l < b(\shp(T)) \quad (c(l, \shp(T)) = 1 \;\rightarrow\; (\forall
    i \in \nat)\,\dt(T)(f(l, \shp(T), i)) = 1)
    \; \rightarrow \\
    (\forall i \in \nat)\,\neg (\dt(T)(g_0(l, \shp(T), i)) = 0 \wedge
    \dt(T)(g_1(l, \shp(T), i)) = 0)
  \end{multline}  
\end{lemma}

\begin{proof}
  We define $b(\shp(\nil))$ to be $0$. We can then take $c,f,g_0,g_1$
  to be anything (e.g. constantly equal to $0$).

  We now deal with the case $T = \maketree(T_i ; \alpha_i)$. We define
  \begin{equation*}
    \label{eq:36}
    b(\shp(\maketree(T_1,\ldots,T_n ; \alpha_1,\ldots,\alpha_n)) :=
    \sum_{i = 1}^n \, b(\shp(T_i)) + n (n - 1)
  \end{equation*}

  Now given $l < b(\shp(T_1,\ldots,T_n ; \alpha_1,\ldots,\alpha_n))$
  we have one of the following two cases (and we can decide which in a
  primitive recursive manner).
  \begin{enumerate}
  \item For some (unique) $0 \leq l_0 < n$ and $0 \leq l_1 < n - 1$, 
    $l = \sum_{i = 1}^n\, b(\shp(T_i)) \;+\; n l_0 + (l_1 - 1)$.
  \item For some $1 \leq k \leq n$ and $0 \leq l' < b(\shp(T_k))$, $l
    = \sum_{i = 1}^{k - 1}\, b(\shp(T_i)) \;+\; l'$, and this is
    unique when we require furthermore that $k$ is the greatest such
    value.
  \end{enumerate}

  For case 1, we take $c(l, S(T)) := 0$. The value of $f$ now makes no
  difference, so we take it to be constantly $0$. Now write $l_1'$ for
  $l_1$ if $l_1 < l_0$ and $l_1 + 1$ if $l_1 \geq l_0$ (so that in any
  case we have $0 \leq l_1' < n$ and $l_0 \neq l_1'$). We define
  \begin{align*}
    g_0(l, \shp(T), i) &:= 2 n i + 2 l_0 \\
    g_1(l, \shp(T), i) &:= 2 n i + 2 l_1'
  \end{align*}
  (This corresponds to ensuring that $\alpha_{l_0} \vee \alpha_{l_1'}
  = 1$)

  For case 2, we define $c(l, S(T)) := 1$. Let $l'$ and $k$ be as in
  the description of case 2. We split into cases on whether or not
  $c(\shp(T_k)) = 1$. If $c(\shp(T_k)) = 1$, then define
  \begin{align*}
    f(l, \shp(T), 2 i) &:= 2 n i + 2 k \\
    f(l, \shp(T), 2 i + 1) &:= 2 n f(l', \shp(T_k), i) + 2k + 1 
  \end{align*}
  If $c(\shp(T_k)) \neq 1$, then define
  \begin{equation*}
    \label{eq:39}
    f(l, \shp(T), i) := 2 n i + k
  \end{equation*}
  
  In either case, we define
  \begin{align*}
    g_0(l, \shp(T), i) &:= 2 n g_0(l',\shp(T_k), i) + 2 k + 1 \\
    g_1(l, \shp(T), i) &:= 2 n g_1(l',\shp(T_k), i) + 2 k + 1       
  \end{align*}  
  (This corresponds to ensuring that if
  $\alpha_k(j) = 1$ for all $j$ then $T_k$ is good.)
\end{proof}

\begin{theorem}
  \label{thm:goodabs}
  Let $\alpha : \nat \rightarrow \nat$. Then the statement ``$f$ is the
  code of a good tree'' is absolute in $\vs$, for any
  proper formal topology $\mathcal{S}$.
\end{theorem}

\begin{proof}
  Note that if $f$ is a primitive recursive function, then the formula
  $f(n) = m$ is equivalent to one built from bounded universal
  quantifiers, conjunctions, $\bot$ and implication, and hence is
  absolute. Note that formula \eqref{eq:35} is built from formulas of
  this form together with function application, bounded universal
  quantification implication and negation. Hence it is absolute. 
  We showed in lemma \ref{lem:goodpr} that the statement that $\alpha$
  codes a good tree is equivalent to this formula and so that is also
  absolute.
\end{proof}

\begin{lemma}
  \label{lem:vgoodpr}
  There are primitive recursive functions $b$ and $f$ such that for
  any $n$-tree $T$, $T$ is very good if and only if there is
  $l < b(\shp(T))$ such that for all $i \in \nat$
  $\dt(T)(f(l, \shp(T), i)) = 1$. Furthermore, assuming Markov's
  principle, if for all $l < b(\shp(T))$, there exists $i \in \nat$
  such that $\dt(T)(f(l, \shp(T), i)) = 0$, then $T$ is not good.
\end{lemma}

\begin{proof}
  For $T = \nil$ we define $b(\shp(T))$ to be $0$, so we can take
  $f(l, \shp(\nil), i)$ to be anything.

  For $T = \maketree(T_1,\ldots,T_n ; \alpha_1,\ldots,\alpha_n)$, we
  define
  \begin{equation*}
    \label{eq:40}
    b(\shp(T)) := \sum_{i = 1}^n b(\shp(T_i))
  \end{equation*}
  Then, note that for $l < b(\shp(T))$, $l$ can be written as
  \begin{equation*}
    \label{eq:41}
    l = \sum_{i = 1}^k b(\shp(T_i)) + l'
  \end{equation*}
  where $0 \leq k < n$ and $0 \leq l' < b(\shp(T_k))$ and  this is
  unique if we require the greatest such $k$.

  Then splitting into cases depending on whether the input to $f$ is
  odd or even, we define
  \begin{align*}
    f(l, \shp(T), 2i) &:= 2ni + k \\
    f(l, \shp(T), 2i + 1) &:= 2 n f(l', \shp(T_k), i) + 2k + 1
  \end{align*}
\end{proof}

\begin{corollary}[$\czf + \markov$]
  \label{cor:nvgoodngood}
  For any $n$-tree $T$, and any list of $n$-trees $T_1,\ldots,T_k$, we
  have
  \begin{enumerate}
  \item If $T$ is good, then the double negation of ``$T$ is very
    good'' is true.
  \item Suppose the following statement is false: $T_i$ is very good
    for every $1 \leq i \leq k$. Then for some $1 \leq i \leq k$,
    $T_i$ is not good.
  \end{enumerate}
\end{corollary}

\begin{proof}
  Note that part 1 follows directly from lemma \ref{lem:vgoodpr}.

  We now show part 2.

  Suppose that it is false that $T_i$ is very good for every $1 \leq i
  \leq k$. We define a finite sequence $\alpha_{0,1}, \ldots,
  \alpha_{0, k} \in \natinfty$ using $f$ from lemma \ref{lem:vgoodpr}
  by,
  \begin{equation*}
    \label{eq:42}
    \alpha_{0, i}(j) := f(0, \shp(T_i), j)
  \end{equation*}

  Note that we cannot have $\alpha_{0, i} = 1$ for all $i$, since then
  each $T_i$ would be very good. Hence by Markov's principle, there is
  some $i_0$ such that $\alpha_{0, i_0} \neq 1$. We then define
  $\alpha_{1, i}$ by
  \begin{equation*}
    \label{eq:43}
    \alpha_{1, i}(j) :=
    \begin{cases}
      \alpha_{0, i}(j) & i \neq i_0 \\
      f(1, \shp(T_i), j) & \text{otherwise}
    \end{cases}
  \end{equation*}
  Then, repeating the same argument as before, we find $i_1$ such that
  $\alpha_{1, i_1} \neq 1$. We continue this process until reach $n$
  such that $i_n = b(\shp(T_{i_n}) - 1$. At this point, we have found
  $j$ such that $f(l, \shp(T_{i_n}), j) \neq 1$ for every $l <
  b(\shp(T_{i_n}))$ and hence can apply lemma \ref{lem:vgoodpr} to
  show that $T_{i_n}$ is not good.
\end{proof}

\section{Some Special Cases of Independence of Premisses}

In this section we define a family of variants of independence of
premisses ($\mathbf{IP}$). The motivation for this it that it allows
us to easily state some special cases of $\mathbf{IP}$ that hold in
certain realizability models and are needed to construct the formal
topologies we will use later.

\begin{definition}
  \label{def:ipgen}
  Let $\Phi(x,y)$ be a formula with only $x$ and $y$ free variables
  and $\Psi(z)$ a formula with only $z$ as a free variable. We will
  think of $\Psi$ as a class, and write $z \in \Psi$ to mean
  $\Psi(z)$. We think of $\Phi(x,y)$ as a class of pairs and write
  $\langle x, y \rangle \in \Phi$ to mean $\Phi(x,y)$.

  Write $\ip{\Phi}{\Psi}$ for the following axiom schema. For any
  formula $\phi$,
  \begin{equation*}
    \label{eq:9}
    \langle x, y \rangle \in \Phi \quad \rightarrow \quad
    ((\forall u \in y)(\exists v \in \Psi)\,\phi) \;
    \rightarrow \;
    ((\forall u \in x)(\exists v \in \Psi)\,(u \in y \rightarrow \phi))
  \end{equation*}
\end{definition}

\begin{lemma}
  \label{lem:retractip}
  Let $X$ and $Y$ be definable sets. By viewing them as classes in the
  usual way, we can define $\ip{\Phi}{X}$ and $\ip{\Phi}{Y}$. If there
  are (provably and definably) functions $f : X \rightarrow Y$ and $g
  : Y \rightarrow X$ such that $f \circ g = 1_Y$, then $\ip{\Phi}{X}$
  implies $\ip{\Phi}{Y}$.
\end{lemma}

\begin{proof}
  We want to show
  \begin{equation*}
    \langle x, y \rangle \in \Phi \quad \rightarrow \quad
    ((\forall u \in y)(\exists v \in Y)\,\phi) \;
    \rightarrow \;
    ((\forall u \in x)(\exists v \in Y)\,(u \in y \rightarrow \phi))
  \end{equation*}
  So assume that $x, y \in \Phi$ and $((\forall u \in y)(\exists v \in
  Y)\,\phi)$. Note that we can define a formula $\phi'(u, w)$
  equivalent to $\phi(u, f(w))$ and show
  \begin{equation*}
    (\forall u \in y)(\exists w \in X)\,\phi'
  \end{equation*}
  This is because for every $u \in y$, we have some $v \in Y$ such
  that $\phi(v)$, but we can then take $w$ to be $g(v)$. Then since
  $f(w) = f(g(v)) = v$, we have $\phi(u, f(w))$.

  Now applying $\ip{\Phi}{X}$, we have
  \begin{equation*}
    (\forall u \in x)(\exists w \in X)\,(u \in y \rightarrow \phi')
  \end{equation*}
  Taking $v$ to be $f(w)$, we have
  \begin{equation*}
    (\forall u \in x)(\exists v \in Y)\,(u \in y \rightarrow \phi)
  \end{equation*}
  But we have now proved $\ip{\Phi}{Y}$, as required.
\end{proof}

\subsection{The Schema $\ipnn$}
\label{sec:schema-ipnn}

We now come to the special cases, $\ipnn$, of $\ip{\Phi}{\Psi}$ that
we will need to construct the formal topologies later.

\begin{definition}
  \label{def:3}
  Let $n \in \nat$. Define $\mathcal{F}_n$ to be the class of pairs
  $\langle x, y \rangle$ where $x$ is of the form
  $\{\alpha_1,\ldots,\alpha_n\}$ where $\alpha_1,\ldots,\alpha_n \in
  \natinfty$ such that for any $1 \leq i \neq j \leq n$, $\alpha_i
  \vee \alpha_j = 1$ and $y = x \cap \{1\}$.

  Then viewing $\baire$ as a class, we define $\ipnn$ according to
  definition \ref{def:ipgen}.
\end{definition}

It is important to note that $\ipnn$ implies several variants, that
will also be used throughout this paper. Where it is clear from
context, we will write that we invoke $\ipnn$ when we actually mean
one of the variants listed below.

\begin{proposition}
  $\ipnn$ implies $\ip{\mathcal{F}_n}{\nat}$,
  $\ip{\mathcal{F}_n}{\operatorname{List}(\baire)}$,
  $\ip{\mathcal{F}_n}{\mathcal{T}_n}$ and
  $\ip{\mathcal{F}_n}{\operatorname{List}(\mathcal{T}_n)}$ where we write
  $\operatorname{List}(X)$ for the set of finite lists of elements of
  $X$ and $\mathcal{T}_n$ to mean the set of $n$-trees.
\end{proposition}

\begin{proof}
  One can easily define suitable functions to apply lemma
  \ref{lem:retractip}. For $n$-trees we use the ``shape and data''
  encoding from definition \ref{def:shpdata}.
\end{proof}

\begin{lemma}
  \label{lem:llpoipnn}
  $\llpo_n$ implies $\ip{\mathcal{F}_n}{\Psi}$ for any $\Psi$ (and in
  particular $\llpo_n$ implies $\ipnn$). 
\end{lemma}

\begin{proof}
  Suppose that $x = \{\alpha_1,\ldots,
  \alpha_n \}$ where $\alpha_i \vee \alpha_j = 1$ for $i \neq j$ and
  such that for all $u \in x \cap \{1\}$ there exists $v \in \Psi$
  such that $\phi(u, v)$.

  By $\llpo_n$, we know that $\alpha_i = 1$ for some $i$. However,
  this implies that $1 \in u \cap \{1\}$, so there must exist $v \in
  \Psi$ such that $\phi(1, v)$. Note that we trivially have that $u =
  1$ implies $\phi(u, v)$, and so we have now proved this instance of
  $\ip{\mathcal{F}_n}{\Psi}$.
\end{proof}

\subsection{$\ipnn$ in $V(\klone)$}

We now check that $\ipnn$ actually holds in the most basic
realizability model for set theory, $V(\klone)$, developed by McCarty
in \cite{mccarty}. The proof uses a key
idea that is already implicit in Lifschitz's original presentation of
Lifschitz realizability \cite{lifschitzrealiz} and also appears the
newer versions by Van Oosten \cite{vanoostenlifschitz}.

\begin{lemma}[$\czf + \markov$]
  \label{lem:ipnnrealized}
  $\ipnn$ holds in $\vkone$. In fact, a more general version
  holds. Let $\Phi$ be the class of pairs $\langle x, y \rangle$ with
  $x$ any subset of $\baire$ and $y = x \cap \{1\}$ (writing $1$ for
  the function constantly equal to $1$). Then $\ip{\Phi}{\baire}$
  holds in $\vkone$.
\end{lemma}

\begin{proof}
  Note firstly that we can show in $\czf$ that for any $f \in \baire$,
  $\neg \neg f = 1$ implies $f = 1$. Hence, we can replace $y$ by $\{f
  \in x \;|\; \neg \neg f = 1 \}$.
  
  We are given $a_0, a_1 \in \klone$ such that
  \begin{align*}
    \label{eq:10}
    a_0 &\Vdash (\forall u \in x)\, u \in \baire \\
    a_1 &\Vdash (\forall u \in x)\; \neg \neg u = 1 \rightarrow
    (\exists v \in \baire)\, \phi
  \end{align*}
  and need to construct computably $b \in \klone$ such that
  \begin{equation*}
    \label{eq:11}
    b \Vdash (\forall u \in x)(\exists v \in \baire)\;
    \neg \neg u = 1 \rightarrow \phi
  \end{equation*}
  Note that for any formula $\psi$, we have $c \Vdash \neg \psi$ for
  some $c \in \klone$ if and only if $c \Vdash \neg \psi$ for
  \emph{every} $c \in \klone$. Hence, if $c \Vdash \neg \neg u = 1$
  for some $c \in \klone$, then $0 \Vdash u = 1$.

  Now let $\langle d, u \rangle \in x$. Note that $(a_0 d)_0$ is a
  code for a total computable function. We define a new computable
  function as follows. Given input $n$, in parallel, run the following
  two algorithms.

  \paragraph{First algorithm:} For each $m$ in turn, evaluate $(a_0
  d)_0 m$. If $(a_0 d)_0 m \neq 1$, then halt and return
  $0$. Otherwise, continue running.

  \paragraph{Second algorithm:} Try to evaluate $a_1 d 0$. If this is
  successful, then try to evaluate $(a_1 d 0)_0 n$. If this is
  successful, then halt and return $(a_1  d 0)_0 n$.

  \paragraph{}
  Let $n \in \klone$. Suppose that neither of these algorithms
  halts. Then in particular, for all $m$, $(a_0 d)_0 m = 1$. However,
  we would then have $0 \Vdash \neg \neg u = 1$ and so $a_1 d 0$ must
  be defined, with $(a_1 d 0)_0$ a total computable function. This
  implies that the second algorithm halts successfully, giving a
  contradiction. Hence by $\markov$ one of the algorithms must halt,
  and so we get a total computable function. Note that we did this
  uniformly in $d$, so in fact we have $b_0 \in \klone$ such that for
  each $\langle d, u \rangle \in x$, $b_0 d$ denotes and is a total
  computable function defined as above.

  Now define $b$ such that for every $d \in \klone$,
  \begin{equation*}
    \label{eq:12}
    b d = \pair (b_0 d) (\lambda z. (a_1 d 0)_1)
  \end{equation*}
  Note first that for any $\langle d, u \rangle \in x$, $b d
  \downarrow$, since $b_0 d \downarrow$ and for any term $t$, $\lambda
  z. t$ denotes (even if $t$ does not). Furthermore, as shown above,
  $(b d)_0$ is always a total computable function. In particular, we
  have $\langle (b d)_0, \overline{(b d)_0} \rangle \in
  \overline{\baire}$, where $\overline{(b d)_0}$ is the function in
  $V(\klone)$ represented by $(b d)_0$, and $\overline{\baire}$ is the
  standard implementation of $\baire$ in $V(\klone)$.

  Now suppose that for some $c \in \klone$, $c \Vdash \neg \neg u =
  1$. In particular, this implies that for every $m$, $(a_0 d)_0 m =
  1$. Then the first algorithm above never halts. Hence we must have
  that for every $n$, $b_0 d n = (a_0 d 0)_0 n$, and so $\overline{b_0
    d} = \overline{(a_0 d 0)_0}$. But, we also have $(a_1 d 0)_1
  \Vdash \phi[v/\overline{(a_0 d 0)_0}]$. Therefore we have
  established that
  \begin{equation*}
    \label{eq:13}
    (b d)_1 c \Vdash \phi[v/\overline{b_0 d}]
  \end{equation*}
  and so
  \begin{equation*}
    \label{eq:14}
    b \Vdash (\forall u \in x)(\exists v \in \baire)\;
    \neg \neg u = 1 \rightarrow \phi
  \end{equation*}
  as required. Finally, note that we constructed $b$ uniformly in $a$,
  so we do indeed have a realizer for the implication
  \begin{equation*}
    \label{eq:15}
    ((\forall u \in y)(\exists v \in \baire)\,\phi) \;
    \rightarrow \;
    ((\forall u \in x)(\exists v \in \baire)\,(u \in y \rightarrow \phi))
  \end{equation*}
\end{proof}

\subsection{$\ipnn$ in Realizability with Truth}

We now do the same thing for realizability with truth. For this to
work we this time need to assume that $\ipnn$ holds already in the
background universe (which was not needed for $V(\klone)$).

\begin{lemma}[$\czf + \markov + \ipnn$]
  \label{lem:iptruth}
  $\ipnn$ holds in the realizability with truth model $\vtruth$
  studied in \cite{rathjen05}.
\end{lemma}

\begin{proof}
  Let $\vtruth$ be the realizability with truth model from
  \cite{rathjen05}. We will construct, for each instance $\psi$ of
  $\ipnn$ a closed application term $t_\psi$ such that $t_\psi
  \Vdashtr \psi$.

  Recall from the proof of lemma \ref{lem:ipnnrealized}, that each
  instance of $\ipnn$ is equivalent to a formula of the following
  form.
  \begin{multline}
    \label{eq:16}
    (\forall x \in \mathcal{F}_n) ((\forall u \in x) \neg \neg u = 1
    \rightarrow (\exists v \in \baire)\,\phi) \; \rightarrow \\
    ((\forall u \in x)(\exists v \in \baire)\,\neg\neg u = 1
    \rightarrow \phi)
  \end{multline}
  Finding a realizer for this formula amounts to
  \begin{enumerate}
  \item Showing that the implication is true
  \item Constructing $a$ such that whenever
    \begin{equation}
      \label{eq:17}
      b \Vdashtr (\forall x \in \mathcal{F}_n) ((\forall u \in x) \neg \neg u = 1
      \rightarrow (\exists v \in \baire)\,\phi)
    \end{equation}
    $a b$ is defined, and
    \begin{equation}
      \label{eq:18}
      a b \Vdashtr (\forall u \in x)(\exists v \in \baire)\,\neg\neg u = 1
      \rightarrow \phi 
    \end{equation}
  \end{enumerate}

  To show 1, we simply apply $\ipnn$ in the background.

  For 2, let $b$ be as in \eqref{eq:17}. We need to construct a
  realizer as in \eqref{eq:18}. Since the formula is of the form
  $(\forall u \in x)\, \psi$, we need to show $(\forall u \in
  x^\circ)(\exists v \in \baire)\,\neg \neg u = 1 \rightarrow
  \phi^\circ$ and construct $a b$ such that for any $\langle d, u
  \rangle \in x$,
  \begin{equation*}
    \label{eq:19a}
    a b d \Vdashtr (\exists v \in \baire)\, \neg \neg u
    = 1 \rightarrow \phi
  \end{equation*}
  For the truth part, we once again apply $\ipnn$ in the
  background. For the realizability part, we follow the same proof as
  for lemma \ref{lem:ipnnrealized} to construct a total computable
  function $f$.

  Finally, we need to construct a realizer for
  \begin{equation*}
    \label{eq:19}
    \neg \neg u = 1 \rightarrow \phi[v/\overline{f}]
  \end{equation*}
  Since, this is an implication, it once again consists of both a
  realizability part and a truth part. However, by \cite[Lemma
  5.10]{rathjen05} we have that if $\neg \neg u^\circ = 1$ is true,
  then $0 \Vdashtr \neg \neg u = 1$. Hence, we can apply the proof
  used in lemma \ref{lem:ipnnrealized} for both parts, and therefore
  the same realizer constructed there still works for this case.
\end{proof}

\begin{theorem}
  \label{thm:ipnep}
  Let $T$ be one of the theories $\czf$, $\czf + \rea$, $\izf$,
  $\izf + \rea$. Then $T + \markov + \ipnn$ has the numerical
  existence property and is closed under Church's rule.
\end{theorem}

\begin{proof}
  Using lemma \ref{lem:iptruth}, the proof of \cite[Theorem
  1.2]{rathjen05} now applies here.
\end{proof}

\subsection{$\ipnn$ in Function Realizability Models}
\label{sec:ipnn-vkltwo}

We now check that the same axioms, $\ipnn$, also hold in function
realizability models.

\begin{lemma}[$\czf + \markov$]
  \label{lem:parallelktwo} There is $\alpha \in \kltwo$ such that the
  following holds. Suppose that $\beta \in \kltwo$ is such that for
  all $\gamma \in \kltwo$ if $\gamma(n) = 1$ for all $n \in \nat$,
  then $\beta \gamma \downarrow$. Then,
  \begin{enumerate}
  \item $\alpha \beta \downarrow$.
  \item For all $\gamma \in \kltwo$, $\alpha \beta \gamma \downarrow$.
  \item If $\gamma(n) = 1$ for all $n \in \nat$, then ($\beta \gamma
    \downarrow$ by assumption and) $\alpha \beta \gamma = \beta
    \gamma$.
  \end{enumerate}
\end{lemma}

\begin{proof}
  We define $\alpha$ so that for each $\beta$, $\alpha \beta$ is as
  follows.
  \begin{equation*}
    \label{eq:2}
    \alpha \beta (\langle n, m_1,\ldots,m_k \rangle) = 
    \begin{cases}
      1 & \text{if } m_i \neq 1 \text{ for some } i \leq
      k \\
      \beta(\langle n, m_1,\ldots,m_k \rangle) & \text{otherwise}
    \end{cases}
  \end{equation*}
  Note that there is such an $\alpha$ since this is clearly continuous
  in $\beta$ and any continuous function is representable in
  $\kltwo$. Also, note that by unfolding the definition of application
  in $\kltwo$ and applying $\markov$ one can show that $\alpha$
  is as required.
\end{proof}

\begin{lemma}[$\czf + \markov$]
  \label{lem:ipnnktworealized}
  Let $\vktwo$ be the function realizability model from
  \cite{rathjenbrouwerian}. Let $\Phi$ be the class of pairs $\langle
  x, y \rangle$ with $x$ any subset of $\baire$ and $y = x \cap \{1\}$
  (writing $1$ for the function constantly equal to $1$). Then
  $\ip{\Phi}{\baire}$ (and hence also $\ip{\mathcal{F}_n}{\baire}$ for
  each $n$) holds in $\vktwo$.
\end{lemma}

\begin{proof}
  One can easily use lemma \ref{lem:parallelktwo} to adapt the proof
  of lemma \ref{lem:ipnnrealized} to work over $\kltwo$.
\end{proof}

\section{The Topological Models $\vln$}

We now define the topological models.

In this section, we will assume a fixed $n$ throughout, and refer to
$n$-trees simply as trees.

\subsection{Definition of $\mathcal{L}_n$}
\label{sec:defin-mathc}

In this section we define the formal topologies that we will use for
the topological models and check that they are in fact formal
topologies. The basic idea is to use the formulation of $\llpo_n$ in
terms of trees to produce the simplest formal topology where $\llpo_n$
holds in the respective topological model, even when it does not hold
in the background universe. This is based on the observation of Van
Oosten in \cite{vanoostentworemarks} that the Lifschitz realizability
topos is the largest subtopos of the effective topos where an axiom
equivalent to $\llpo$ in the presence of Church's thesis
holds.

\begin{definition}
  Let $T$ be a tree. Then we define the \emph{cover from $T$}, $\cov(T)
  \subseteq \{0\}$, inductively as follows.
  \begin{enumerate}
  \item $\cov(\nil) = \{0\}$
  \item $\cov(\maketree(T_i ; \alpha_i)) = \bigcup_{i = 1}^n \{ 0 \in
    \cov(T_i) \;|\; \alpha_i = 1 \}$
  \end{enumerate}
\end{definition}

\begin{lemma}
  \label{lem:covvgood}
  Let $T$ be a good tree. Then $0 \in \cov(T)$ if and only if $T$ is
  very good.
\end{lemma}

\begin{proof}
  We show this by induction on trees.

  For $T = \nil$, we have both $0 \in \cov(T)$ and $T$ is very good,
  so the result is clear.

  Now suppose that $T = \maketree(T_i ; \alpha_i)$. If $T$ is very
  good then for some $i$, $\alpha_i = 1$ and $T_i$ is very
  good. However, if $T_i$ is very good, then $0 \in \cov(T_i)$ by the
  induction hypothesis, and so, we have $0 \in \cov(T)$. We have shown
  that if $T$ is very good then $0 \in \cov(T)$. Now suppose that $0
  \in \cov(T)$. Then for some $i$, $\alpha_i = 1$ and $0 \in
  \cov(T_i)$. The latter implies $T_i$ is very good by the induction
  hypothesis, and so by the former $T$ is very good, as required.
\end{proof}

\begin{proposition}[$\czf + \markov$]
  \label{prop:dncovinh}
  Let $T$ be a good tree. Then we have $\neg \neg 0 \in \cov{T}$.
\end{proposition}

\begin{proof}
  Suppose that $T$ is a good tree and that $0 \notin \cov(T)$. Since
  $0 \notin \cov(T)$, we have by lemma \ref{lem:covvgood} that $T$ is
  not very good. Then by corollary \ref{cor:nvgoodngood} we have that $T$
  is not good, giving us a contradiction. Hence we have $\neg \neg 0
  \in \cov{T}$ as required.
\end{proof}

\begin{definition}
  Let $S, \leq$ be the poset with $S = \{0\}$. Define the relation
  $\triangleleft$ as follows. $0 \triangleleft p$ precisely if
  $\cov(T) \subseteq p$ for some good tree, $T$. Write $\mathcal{L}_n$
  for the tuple $\langle S, \leq, \triangleleft \rangle$ (we will show
  that this is a formal topology).
\end{definition}

\begin{lemma}[$\czf + \markov$]
  \label{lem:nalmostformaltop}
  $\mathcal{L}_n$ satisfies axioms 1, 2 and 4 in the definition of
  formal topology.
\end{lemma}

\begin{proof}
  1 and 2 are clear. It remains to prove 4, that is, that whenever $0
  \triangleleft p$ and $0 \triangleleft q$, we have $0 \triangleleft p
  \cap q$.

  Fix a good tree, $T$. We will show by induction that for any tree
  $S$, there is a tree $R$ such that $\cov(R) \subseteq \cov(T) \cap
  \cov(S)$, and that if $S$ is good then $R$ is also good.

  For $S = \nil$, we just take $R$ to be $T$.

  Now suppose that $S = \maketree(S_i ; \alpha_i)$. Then we have for
  each $i$, a tree $R_i$ such that $\cov(R_i) \subseteq \cov(T) \cap
  \cov(S_i)$ and $R_i$ is good if $S_i$ is good. Define $R$ to be the
  tree $\maketree(R_i ; \alpha_i)$. Suppose that $0 \in \cov(R)$. Then
  for some $1 \leq i \leq n$ we must have $\alpha_i = 1$ and $0 \in
  \cov(R_i)$. Since $\cov(R_i) \subseteq \cov(T) \cap \cov(S_i)$, we
  also have $0 \in \cov(S_i)$ and $0 \in \cov(T)$. But, now recalling
  that $\alpha_i = 1$, the former implies $0 \in \cov(S)$. Hence,
  $\cov(R) \subseteq \cov(T) \cap \cov(S)$.

  Now suppose that $S$ is good. Then we have that for any $1 \leq i
  \neq j \leq n$, $\alpha_i \vee \alpha_j = 1$. Also, for any $i$, if
  $\alpha_i = 1$, then $S_i$ is good. But this then implies that $R_i$
  is good. Hence $R$ is also good.

  We can now easily deduce axiom 4.
\end{proof}

\begin{theorem}[$\czf + \markov + \ip{\mathcal{F}_n}{\baire}$]
  $\mathcal{L}_n$ is a formal topology.
\end{theorem}

\begin{proof}
  We have already shown in lemma \ref{lem:nalmostformaltop} that
  axioms 1, 2 and 4 hold. It remains to show that axiom 3 holds. That
  is, whenever $0 \triangleleft p$ and $p \triangleleft q$, we have $0
  \triangleleft q$.

  Fix $q \subseteq \{0\}$. We show the following by induction on
  trees. Let $T$ be a tree. Suppose that $T$ is good and whenever $0
  \in \cov(T)$ we have $0 \triangleleft q$. Then there is a good tree
  $S$ such that $\cov(S) \subseteq q$.

  First assume $T = \nil$. Then $0 \in \cov(T)$, and so we have $0
  \triangleleft q$. Let $S$ be any good tree such that $\cov(S)
  \subseteq q$.

  Now assume that $T = \maketree(T_i ; \alpha_i)$. Assume that $T$ is
  good and whenever $0 \in \cov(T)$ we have $0 \triangleleft q$. Since
  $T$ is good, we have that for any $1 \leq i \neq j \leq n$,
  $\alpha_i \vee \alpha_j = 1$. Let $1 \leq i \leq n$ be such that
  $\alpha_i = 1$. Then $T_i$ is good, and $\cov(T_i) \subseteq
  \cov(T)$. The latter implies that whenever $0 \in \cov(T_i)$ we have
  $0 \triangleleft q$ and so we may apply the induction hypothesis, to
  show there exists $S$ such that $\cov(S) \subseteq q$.

  However, we can now apply $\ip{\mathcal{F}_n}{\baire}$ to find for
  each $1 \leq i \leq n$, a tree $S_i$ such that if $\alpha_i = 1$
  then $S_i$ is good and $\cov(S_i) \subseteq q$. Define $S$ to be
  $\maketree(S_i ; \alpha_i)$. Then whenever $i$ is such that
  $\alpha_i = 1$, we have that $S_i$ is good, and so $S$ must be
  good. Suppose that $0 \in \cov(S)$. Then for some $i$ we have
  $\alpha_i = 1$ and $0 \in \cov(S_i)$. Hence also $0 \in q$. But we
  have now shown $\cov(S) \subseteq q$ as required.
\end{proof}

\subsection{Some Basic Properties of $\mathcal{L}_n$ and
  $\vln$}
\label{sec:some-basic-prop}

\begin{lemma}[$\czf + \markov$]
  If $\llpo_n$ is true, then we have
  \label{lem:degeneratevln}
  \begin{enumerate}
  \item $\vln$ is isomorphic to the class of all sets, $V$.
  \item $\vln \models \phi$ if and only if $\phi$ is true.
  \end{enumerate}
\end{lemma}

\begin{proof}
  By $\llpo_n$, we know that every good $n$-tree is very good.  Hence,
  in this case $\mathcal{L}_n$ reduces to the trivial formal topology,
  where for every $p \subseteq \{0\}$, $0 \triangleleft p$ if and only
  if $0 \in p$. The result clearly follows.
\end{proof}

\begin{lemma}[$\czf + \markov + \ipnn$]
  \label{lem:lncompact}
  For each $j \in \nat$, let $p_j$ be a subset of $\{0\}$. Suppose
  that $0 \triangleleft \bigcup_{j \in \nat} p_j$. Then there is some
  finite set $J \subseteq \nat$ such that $0 \triangleleft \bigcup_{j
    \in J} p_j$. (That is, $\mathcal{L}_n$ is \emph{countably
    compact}.)
\end{lemma}

\begin{proof}
  We show by induction on trees, that for every tree $T$, if $T$ is
  good and $\cov(T) \subseteq \bigcup_{j \in \nat} p_j$ then there
  exists a finite set $J \subseteq \nat$ and another good tree $S$
  such that $\cov(S) \subseteq \bigcup_{j \in J} p_j$.

  For $T = \nil$, we have $0 \in \bigcup_{j \in \nat} p_j$ and so for
  some $j \in \nat$, $0 \in p_j$. Hence we can just take $J := \{j\}$
  and $S = \nil$.

  Now suppose $T = \maketree(T_i ; \alpha_i)$. Note that if $1 \leq i
  \leq n$ is such that $\alpha_i = 1$, then $T_i$ is good and
  $\cov(T_i) \subseteq \cov(T) \subseteq \bigcup_{j \in \nat} p_j$. So
  by the induction hypothesis, there is a finite set $J$ and a good
  tree $S$ such that $\cov(S) \subseteq \bigcup_{j \in J} p_j$. Hence
  we can apply $\ipnn$ to find for each $1 \leq i \leq n$, a finite
  set $J_i \subseteq \nat$ and a tree $S_i$ such that if $\alpha_i =
  1$ then $S_i$ is good and $\cov(S_i) \subseteq \bigcup_{j \in J_i}
  p_j$. We then take $J := \bigcup_{i = 1}^n J_i$ and $S :=
  \maketree(S_i ; \alpha_i)$ and note these are as required.
\end{proof}

\begin{lemma}[$\czf + \markov + \ipnn$]
  \label{lem:listwitnesses}
  Suppose that $\vln \models (\exists j \in \nat)\,\phi(j)$.
  Then there is some finite $J \subseteq \nat$ such that
  $\vln \models (\exists j \in \hat{J})\,\phi(j)$.
\end{lemma}

\begin{proof}
  Apply lemma \ref{lem:lncompact} with $p_j := \llbracket
  \phi(\hat{j}) \rrbracket$ for $j \in \nat$.
\end{proof}

The following lemma will be key to showing later that certain choice
axioms and existence properties hold. It appears to be related to the
constructions developed by Lee and Van Oosten in \cite[Sections 4 and
5]{leevanoosten}. We will return to this point in section
\ref{sec:conn-topos-theory}.
\begin{lemma}[$\czf + \markov + \ip{\mathcal{F}_n}{\baire}$]
  \label{lem:buildwitness}
  Let $1 \leq k < n$ and for each $j \in \nat$, let $p_j$ be a subset
  of $\{0\}$. Suppose that $0 \triangleleft \bigcup_{j \in \nat} p_j$
  (relative to $\mathcal{L}_n$). Suppose further that for every $J
  \subseteq \nat$ such that $J$ is finite and $|J| > k$ we have
  $\bigcap_{j \in J} p_j = \emptyset$.

  Then for some $j \in \nat$ there exists a good $\lceil \frac{n}{k}
  \rceil$-tree, $S$ such that $\cov(S) \subseteq p_j$ (where $\lceil
  \frac{n}{k} \rceil$ means round up $\frac{n}{k}$ to the next
  integer).
\end{lemma}

\begin{proof}
  We show by induction on trees that for every $n$-tree, $T$, if $T$
  is good and $\cov(T) \subseteq \bigcup_j p_j$, then there exists $j
  \in \nat$ and an $\lceil \frac{n}{k} \rceil$-tree $S$ such that
  $\cov(S) \subseteq p_j$.

  For $T = \nil$, we have $0 \in \bigcup_{j \in \nat} p_j$. Hence for
  some $j \in \nat$ we in fact have $0 \in p_j$. We can then take $S$ to
  be $\nil$.

  Now suppose that $T = \maketree(T_i ; \alpha_i)$.

  Suppose that $\alpha_i = 1$. Then $T_i$ is good and $\cov(T_i)
  \subseteq \bigcup_i p_i$. So there exist $j \in \nat$ and $S$ a good
  $\lceil \frac{n}{k} \rceil$-tree such that $\cov(S) \subseteq p_j$.

  Hence we can apply $\ip{\mathcal{F}_n}{\baire}$ to find for each $1
  \leq i \leq n$, $j_i \in \nat$ and an $\lceil \frac{n}{k}
  \rceil$-tree $S_i$ such that if $\alpha_i = 1$ then $S_i$ is good
  and $\cov(S_i) \subseteq p_{j_i}$.

  Now suppose that $|\{ j_i \;|\; 1 \leq i \leq n \}| > k$. Let $I
  \subseteq \{1,\ldots,n\}$ be such that $|I| = |\{ j_i \;|\; 1 \leq i
  \leq n \}| = |\{j_i \;|\; i \in I \}|$ (which exists by finite
  choice and decidability of equality for $\nat$). By assumption,
  $\bigcap_{i \in I} p_i = \emptyset$. Suppose that for all $i \in I$,
  $\alpha_i = 1$. Then we would have that each $S_i$ is good but
  $\bigcap_{i \in I} \cov(S_i) = \emptyset$, giving a contradiction by
  corollary \ref{cor:nvgoodngood} and lemma \ref{lem:covvgood}. Hence
  by lemma \ref{lem:meetnotone}, for some $i$, $\alpha_i \neq 1$. Let
  $i' \in I \setminus \{i\}$. Since $\alpha_i \neq 1$, we vacuously
  have $\alpha_i = 1$ implies that $\cov(S_i) \subseteq
  p_{j_{i'}}$. Hence we may ``replace'' $j_i$ with $j_{i'}$.

  By repeating the above argument we may assume without loss of
  generality that in fact
  \begin{equation*}
    \label{eq:22}
    |\{ j_i \;|\; 1 \leq i \leq n \}| \leq k
  \end{equation*}
  Write $J$ for the set $\{ j_i \;|\; 1 \leq i \leq n \}$.
  
  Now note that we have
  \begin{equation*}
    \sum_{j \in J} | \{i \;|\; j_i = j \} | \; = \; n
  \end{equation*}
  Note that if $l \in \nat$ is such that $l < \lceil \frac{n}{k}
  \rceil$, then $l < \frac{n}{k}$. To show this, see that we can find
  $p,q \in \mathbb{N}$ with $0 \leq q < k$ such that $n = p k + q$ by
  Euclid's algorithm. We can then split into cases depending on
  whether or not $q = 0$, by decidability of equality for
  $\mathbb{N}$. If $q = 0$, then $l < \lceil \frac{n}{k} \rceil =
  \frac{n}{k}$. If $q > 0$, then $l \leq \lceil \frac{n}{k} \rceil - 1
  < \frac{n}{k}$. So in either case $l < \frac{n}{k}$.
  
  Hence, if we had $| \{i \;|\; j_i = j \} | < \lceil \frac{n}{k}
  \rceil$ for all $j \in J$, this would imply $\sum_{j \in J} | \{i
  \;|\; j_i = j \} | < \frac{n}{k}. k = n$, giving a
  contradiction. Hence, for some $j \in J$ we must have $| \{i \;|\;
  j_i = j \} | \geq \lceil \frac{n}{k} \rceil$. Choose such a $j$, and
  $I \subseteq \{i \;|\; j_i = j \}$ with $|I| = \lceil \frac{n}{k}
  \rceil$ and an enumeration of $I$. Then let $S$ be the $\lceil
  \frac{n}{k} \rceil$-tree $\maketree( (\alpha_i)_{i \in I} ;
  (S_{i})_{i \in I})$. Since $T$ is good and $S_i$ is good when
  $\alpha_i = 1$, $S$ must also be good. Now suppose $0 \in
  \cov(S)$. This implies that for some $i \in I$, $\alpha_i = 1$ and
  $0 \in S_i$. But then also $0 \in p_j$. So $\cov(S) \subseteq p_j$
  as required.
\end{proof}

\begin{remark}
  Note that in the above lemma we do not have $0 \triangleleft p_j$
  relative to $\mathcal{L}_n$, because we require a good $n$-tree $S$,
  such that $\cov(S) \subseteq p_j$, but have only a good $\lceil
  \frac{n}{k} \rceil$-tree. We do however have $\neg \neg 0 \in p_j$.
\end{remark}

\begin{lemma}[$\czf + \markov + \ip{\mathcal{F}_n}{\baire}$]
  \label{lem:getuniquewitness}
  Suppose that for each $j \in \nat$, $p_j$ is a subset of $\{0\}$
  such that $0 \triangleleft \bigcup_{j \in \nat} p_j$ and that for
  all $j \neq j' \in \nat$ we have $p_j \cap p_{j'} = \emptyset$. Then
  for some (necessarily unique) $j \in \nat$ we have $0 \in p_j$.
\end{lemma}

\begin{proof}
  This is a special case of lemma \ref{lem:buildwitness} with $k = 1$.
\end{proof}

\begin{lemma}[$\czf + \markov + \ipnn$]
  \label{lem:functionabs}
  Suppose that $\vln \models f \in \baire$. Then for some $g : \nat
  \rightarrow \nat$, $\vln \models f = \hat{g}$.
\end{lemma}

\begin{proof}
  We define $g : \nat \rightarrow \nat$ as follows. Let $n \in
  \nat$.
  For each $m \in \nat$, set
  $p_m := \llbracket f(\hat{n}) = \hat{m} \rrbracket$. Note that for
  $m \neq m'$, we have $p_m \cap p_{m'} = \emptyset$, so we can apply
  lemma \ref{lem:getuniquewitness} to find $m$ such that
  $0 \in \llbracket f(\hat{n}) = \hat{m} \rrbracket$. We take $g(n)$
  to be this $m$.

  Note that by construction we have $\vln \models (\forall n \in
  \nat)\,\hat{g}(n) = f(m)$, and so $\vln \models \hat{g} = f$.
\end{proof}

\begin{lemma}[$\czf + \markov + \ipnn$]
  \label{lem:function2abs}
  Suppose that $\vln \models F : \baire \rightarrow \nat$. Then for
  some $G : \baire \rightarrow \nat$, $\vln \models F = \hat{G}$.
\end{lemma}

\begin{proof}
  First note that by lemma \ref{lem:functionabs} we can show that
  $\baire$ is absolute, in the sense that in $\vln$ we can show that
  $\hat{\baire}$ is the set of functions $\nat \rightarrow
  \nat$. However, we can now apply the same proof as in lemma
  \ref{lem:functionabs} to get the result.
\end{proof}

\begin{lemma}[$\czf + \markov + \ipnn$]
  \label{lem:vlnmarkov}
  \begin{equation*}
    \label{eq:52}
    \vln \models \markov
  \end{equation*}
\end{lemma}

\begin{proof}
  Suppose that $f \in \vln$ is such that
  $\vln \models f \in 2^\nat \;\wedge\; \neg \neg (\exists x \in
  \nat)\,f(x) = 1$.
  Then by lemma \ref{lem:functionabs} there is
  $g : \nat \rightarrow 2$ such that $\vln \models \hat{g} = f$. Note
  that $\neg \neg (\exists x \in \nat)\,\hat{g}(x) = 1$ is equivalent to
  $\neg (\forall x \in \nat)\,\hat{g}(x) = 0$ and so is absolute. Hence we
  can apply $\markov$ in the background to find $m \in \nat$ such that
  $g(m) = 1$. But then $\vln \models (\exists x \in \nat)\,f(x) =
  1$. Therefore $\markov$ holds in $\vln$.
\end{proof}

\subsection{$\llpo_n$ in $\vln$}
\label{sec:llpo_n-vln}

The motivation for the definition of $\mathcal{L}_n$ was to try to
write down the simplest topology where $\llpo_n$ holds in the
topological model. We now check that in fact it really is the case
that $\llpo_n$ holds in $\vln$. Note that we don't need to assume
$\llpo_n$ holds in the background for this to work, although we did
need $\ipnn$, even just to construct the topological model.

\begin{lemma}[$\czf + \markov + \ip{\mathcal{F}_n}{\baire}$]
  \label{lem:llponsound}
  \begin{equation*}
    \label{eq:29}
    \vln \models \llpo_n
  \end{equation*}
\end{lemma}

\begin{proof}
  Suppose that $f \in \vln$ is such that internally in $\vln$, $f$ is
  a function $\nat \rightarrow 2$ such that $f(i) = 1$ for at most one
  $i$. Then by lemma \ref{lem:functionabs} there must be some (unique)
  $g : \nat \rightarrow 2$ such that $\vln \models \hat g = f$. Then
  by lemma \ref{lem:abslem} we must have that also $g(i) = 1$ for at
  most one $i$. We now define a tree by setting for $1 \leq k \leq n$,
  \begin{equation*}
    \label{eq:30}
    \alpha_k(i) := 1 - \max_{i' \leq i} (g(n i' + (k - 1)))
  \end{equation*}
  and then define
  \begin{equation*}
    \label{eq:31}
    T := \maketree(\nil,\ldots,\nil ; \alpha_1, \ldots, \alpha_n)
  \end{equation*}
  We clearly have that $T$ is a good tree and by lemma
  \ref{lem:abslem} we know
  \begin{equation*}
    \label{eq:32}
    \cov(T) \subseteq \bigcup_{1 \leq k \leq n} \llbracket (\forall x
    \in \nat)\,f(x n + (\hat{k} - 1)) = 0 \rrbracket
  \end{equation*}
  Hence
  \begin{equation*}
    \label{eq:33}
    \vln \models \bigvee_{1 \leq k \leq n} (\forall x \in \nat)\,f(x n
    + (k - 1)) = 0
  \end{equation*}
  But we now have that $\vln \models \llpo_n$ as required.  
\end{proof}

\subsection{Bounded Existential Formulas and Countable Choice in
  $\vln$}
\label{sec:countable-choice-vln}

Although countable choice fails in each $\vln$, there are weaker
variants that we define below that do hold. To formulate them, we
first define some notation for certain bounded existential formulas.

\begin{definition}
  Let $\phi$ be a formula. We write $(\exlte{n} x) \, \phi$ as
  shorthand for the following formula.
  \begin{equation*}
    (\exists x \in \nat)\,\phi \quad \wedge \quad
    (\forall x_1,\ldots,x_{n + 1} \in \nat) \left( \bigwedge_{i \neq j}(x_i
      \neq x_j) \; \rightarrow \neg \bigwedge_i \phi(x_i) \right)
  \end{equation*}
  Informally, this says that there exists a witness of $\phi(x)$ in
  $\nat$, but given any $X \subseteq \nat$ with $|X| = n + 1$ it is
  false that every element of $X$ is a witness of $\phi(x)$. In other
  words $\phi(x)$ has at least one, but at most $n$ witnesses.
\end{definition}

\begin{definition}
  \label{def:4}
  We define the following variants of the axiom of choice. Let $X$ be
  any set.
  \begin{enumerate}
  \item Write $\acx{k}{}$ for the following principle. Let $\phi(x,
    y)$ be a bounded formula (that may have parameters). Suppose that
    we have $(\forall x \in X)(\exlte{k} y)\,\phi(x,y)$. Then there
    is a function $f : X \rightarrow \nat$ such that for every $x
    \in X$, $\phi(x, f(x))$.
  \item Write $\acx{k}{m}$ for the following principle. Let $\phi(x,
    y)$ be a bounded formula (that may have parameters). Suppose that
    we have $(\forall x \in X)(\exlte{k} y)\,\phi(x,y)$. Then there
    is a function $f : X \rightarrow \nat$ such that for every $x
    \in X$, there is a good $m$-tree, $T$ such that if $T$ is very
    good then $\phi(x, f(x))$.
  \item Write $\acxdn{k}$ for the following principle. Let $\phi(x,
    y)$ be a bounded formula. Suppose that we have $(\forall x \in
    X)(\exlte{k} y)\,\phi(x,y)$. Then there is a function $f : X
    \rightarrow \nat$ such that for all $x \in X$, $\neg \neg
    \phi(x, f(x))$.
  \end{enumerate}
\end{definition}

\begin{proposition}[$\czf + \markov$]
  \label{prop:acimpl}
  Let $X$ be any set. For all $m, k \in \nat$ with $m, k \geq 2$, and
  all $m' \leq m$,
  \[
  \acx{k}{} \; \Rightarrow \; \acx{k}{m} \; 
  \Rightarrow \; \acx{k}{m'} \; \Rightarrow \; \acxdn{k} 
  \]
\end{proposition}

\begin{proof}
  For $(\acx{k}{} \Rightarrow \acx{k}{m})$, note that $\acx{k}{m}$
  is easily a special case of $\acx{k}{}$.

  For $(\acx{k}{m} \Rightarrow \acx{k}{m'})$, given any good $m$-tree
  $T$, we can generate a good $m'$-tree by ``choosing $m'$ branches at
  each level.''

  For $(\acx{k}{m'} \Rightarrow \acxdn{k})$, we just apply corollary
  \ref{cor:nvgoodngood}.
\end{proof}

\begin{lemma}[$\czf + \markov + \ip{\mathcal{F}_n}{\baire} +
  \choice_{\nat, \nat}$]
  \label{lem:vlncc}
  Let $n, k \in \nat$ and $2 \leq k < n$. Then
  \begin{equation*}
    \label{eq:25}
    \vln \models \acn{k}{\lceil \frac{n}{k} \rceil}
  \end{equation*}
\end{lemma}

\begin{proof}
  Let $x \in \nat$ and suppose that $0 \in \llbracket (\exlte{k}
  y)\,\phi(\hat x, y) \rrbracket$. Then we have by unfolding the
  interpretation of formulas in $\vln$ and the definition of
  $\exlte{k}$ that,
  \begin{equation*}
    0 \triangleleft \bigcup_{i \in \nat} \llbracket \phi(\hat x, \hat{i}) \rrbracket
  \end{equation*}
  and for every list $i_1,\ldots,i_{k + 1}$
  \begin{equation*}
    \label{eq:26}
    \bigcap_{1 \leq j \leq k + 1} \llbracket \phi(\hat x, \hat{i_j}) \rrbracket \;=\; \emptyset
  \end{equation*}

  Hence, applying lemma \ref{lem:buildwitness} with $p_i := \llbracket
  \phi(\hat x, \hat{i}) \rrbracket$, we have that for every $x \in
  \nat$ there exists $y \in \nat$ and a good $\lceil \frac{n}{k}
  \rceil$-tree $S$ such that if $S$ is very good then $0 \in
  \llbracket \phi(\hat x, \hat y) \rrbracket$.

  Now applying $\choice_{\nat,\nat}$ we get a choice function $f :
  \nat \rightarrow \nat$. That is, for every $x \in \nat$, there
  exists a good $\lceil \frac{n}{k} \rceil$-tree $S$ such that if $S$
  is very good then $0 \in \llbracket \phi(\hat x, \hat{f(x)})
  \rrbracket$. For each $x \in \nat$, let $g \in \baire$ be a code for
  the tree $S$ as above. Then the statement that $g$ codes a good tree
  is absolute by theorem \ref{thm:goodabs}, so also holds internally.

  Also, the statement that $g$ codes a very good tree is equivalent to
  a formula of the form $(\exists x \in \nat)\,\psi(x)$, where $\psi$
  is negative by lemma \ref{lem:vgoodpr}. Hence by lemma
  \ref{lem:disjexcriteria} the statement ``$\hat g$ codes a very
  good tree implies $ \phi(\hat{x}, \hat{f(x)})$'' must also hold
  internally.
\end{proof}

Finally, we define another variant of choice that will also
hold in our model. This will be denoted \emph{Herbrand
  choice}, since it also holds in the Herbrand topos
developed by Van den Berg in \cite{vdbergherbrandtop}.

\begin{definition}
  We refer to the following principle as $\hacx$ or
  \emph{Herbrand countable choice}. Let $\phi(x, y)$ be a bounded
  formula (that may have parameters). Suppose that we have
  $(\forall x \in X)(\exists y \in \nat)\phi(x, y)$. Then there
  exists a function $f$ from $X$ to the set of finite subsets of
  $\nat$, $\powset_{\mathrm{fin}(\nat)}$, such that for all
  $x \in X$ there exists $m \in f(x)$ such that $\phi(x, m)$.
\end{definition}

One can easily show $\hacx$ can be alternatively formulated as
follows.
\begin{proposition}[$\czf$]
  \label{prop:hebrandboundeddef}
  $\hacx$ is true if and only if the following holds. Suppose
  that we have $(\forall x \in X)(\exists y \in \nat)\phi(x,
  y)$. Then there exists a function $f \colon X \to \nat$ such that
  for all $x \in X$ there exists $m < f(x)$ such that $\phi(x,
  m)$.
\end{proposition}

\begin{lemma}[$\czf + \markov + \ip{\mathcal{F}_n}{\baire} +
  \choice_{\nat, \nat}$]
  \label{lem:hacnnsound}
  \begin{equation*}
    \vln \models \hacn
  \end{equation*}
\end{lemma}

\begin{proof}
  Suppose that
  $\vln \models (\forall x \in \nat)(\exists y \in \nat)\phi(x,
  y)$. Then for every $n \in \nat$, we have
  $\vln \models (\exists y \in \nat)\,\phi(\hat{n}, y)$. By lemma
  \ref{lem:listwitnesses} there exists a finite set $J \subseteq \nat$
  such that $\vln \models (\exists y \in \hat{J})\,\phi(\hat{n},
  y)$. Hence also there exists $N \in \nat$ such that
  $\vln \models (\exists y < \hat{N})\,\phi(\hat{n}, y)$. By
  $\choice_{\nat, \nat}$, we deduce that there is a function
  $f \colon \nat \to \nat$ such that for all $n \in \nat$,
  $\vln \models (\exists y < \widehat{f(n)})\,\phi(\hat{n},
  y)$. Finally by absoluteness, we deduce
  $\vln \models (\forall x \in \nat)(\exists y < \hat{f}(x))\,
  \phi(x, y)$, and thereby $\vln \models \hacn$.
\end{proof}

\section{Applications}
\label{sec:applications}

\subsection{Consistency of Church's Thesis with $\llpo_n$}
\label{sec:cons-churchs-thes}

A hallmark of Lifschitz realizability, from Lifschitz's original model
for arithmetic in \cite{lifschitzrealiz} onwards is that it satisfies
both Church's thesis and $\llpo$. We will recover the result from
\cite{rathjenchen} that Church's thesis and $\llpo$ are compatible
over $\izf$. Moreover, we will show something even stronger.  Certain
variants of the axiom of countable choice are compatible with Church's
thesis and $\llpo$, and as $n$ increases, we can show that successively
stronger forms of countable choice are compatible with Church's thesis
and $\llpo_n$.

\begin{lemma}[$\czf + \markov + \churchu$]
  \label{lem:vlnchurchu}
  \begin{equation*}
    \label{eq:56}
    \vln \models \churchu
  \end{equation*}
\end{lemma}

\begin{proof}
  By lemma \ref{lem:functionabs} it suffices to show that for every $f
  \in \baire$, the statement that $f$ is computable holds in
  $\vln$. For any $f$, we have by applying $\churchu$ in the
  background that there exists $e \in \nat$ such that $f = \{e\}$. For
  every $i \in \nat$, the statement that $f(i) = \{e\}(i)$ is of the
  form $(\exists x \in \nat)\,\phi(x)$ where $\phi$ is primitive
  recursive. Since this holds in the background universe we must also
  have for each $i$, $\vln \models \hat{f}(\hat{i}) =
  \{\hat{e}\}(\hat{i})$. Therefore $\vln \models (\forall x \in
  \nat)\,\hat{f}(x) = \{\hat{e}\}(x)$. Therefore $\vln \models
  \churchu$ as required.
\end{proof}

\begin{theorem}
  \label{thm:ctllpocon}
  Assume that $\czf$ is consistent. Then for each $n \in \nat$, the
  following theory is consistent.
  \[
  \czf + \markov + \llpo_n +  \bigwedge_{2 \leq k < n} \acn{k}{\lceil \frac{n}{k}
    \rceil} + \hacn + \churchu
  \]
  Assume that $\izf$ is consistent. Then for each $n \in \nat$, the
  following theory is consistent.
  \[
  \izf + \markov + \llpo_n +  \bigwedge_{2 \leq k < n} \acn{k}{\lceil \frac{n}{k}
    \rceil} + \hacn + \churchu
  \]  
\end{theorem}

\begin{proof}
  Let $T$ be either $\czf$ or $\izf$ and assume that $T$ is
  consistent. It is already known that in both cases $\markov$ does
  not change the consistency strength. ($\izf$ is the same consistency
  strength as $\mathbf{ZF}$ by the main result in \cite{friedman73}
  and $\czf$ is the same consistency strength as $\czf + \mathbf{LPO}$
  by \cite{rathjenlpo})
  
  So we have that $T + \markov$ is consistent. Then so is the theory
  $T + \markov + \church + \ipnn$ by working in the McCarty
  realizability model $V(\klone)$ and using the main results in
  \cite{mccarty} and \cite{rathjen06} together with lemma
  \ref{lem:ipnnrealized}.

  However we now get the result by building the model $\vln$ in
  $T + \markov + \church + \ipnn$ and applying lemmas
  \ref{lem:vlnmarkov}, \ref{lem:llponsound}, \ref{lem:vlncc},
  \ref{lem:hacnnsound} and
  \ref{lem:vlnchurchu}.
\end{proof}

In \cite{richmanllpon}, Richman gave a proof in Bishop style
constructive mathematics that for each $n$, $\llpo_n$ is inconsistent
with the statement that all functions are computable (that in fact
this is even true for $\llpo_\infty$). Richman's argument does not
hold in $\czf$ or even $\izf$, as is already clear from the earlier
Lifschitz realizability model in \cite{rathjenchen}. However, it turns
out that the only obstacle is an implicit use of countable choice, and
one can use $\acndn{n}$ to carry out Richman's argument, as follows.
\begin{theorem}
  \label{thm:richmansthm}
  For each $n \in \nat$, the following theory is inconsistent.
  \begin{equation*}
    \label{eq:21}
    \czf + \llpo_n + \churchu + \acndn{n}
  \end{equation*}
\end{theorem}

\begin{proof}
  For each $i, j \in \nat$ with $j < n$, we define $\alpha_{i, j} \in
  \natinfty$ as follows. $\alpha_{i,j}(k)$ is equal to $0$ if the
  $i$th Turing machine with input $i$ has halted by stage $k$ with
  output $j$, and $\alpha_{i, j}(k)$ is equal to $1$ otherwise.

  Note that for any $i$ and for any $j, j' < n$ with $j \neq j'$ we
  have $\alpha_{i, j} \vee \alpha_{i, j'} = 1$ (since the $i$th Turing
  machine on input $i$ can have at most $1$ output). Hence we can
  apply $\llpo_n$ to show that for some $j < n$, $\alpha_{i, j} = 1$.

  Now we can apply $\acndn{n}$ to find a function $f : \nat
  \rightarrow n$ such that for each $i$, $\neg \neg \alpha_{i, f(i)} =
  1$. (In fact this implies that $\alpha_{i, f(i)} = 1$, but we don't
  need this.)

  Now apply $\churchu$ to find $e \in \nat$ such that for all $i$,
  $\{e\}(i) = f(i)$. In particular, the $e$th Turing machine with
  input $e$ halts with output $f(e)$. Hence, for sufficiently large
  $k$ we have $\alpha_{e, f(e)} (k) = 0$ and so $\alpha_{e, f(e)} \neq
  1$. However, $f(e)$ was chosen so that $\neg \neg \alpha_{e, f(e)} =
  1$. Therefore we get a contradiction, as required.
\end{proof}

Hendtlass and Lubarsky showed in \cite{hendtlasslubarsky} that
$\llpo_{n + 1}$ is independent of $\llpo_n$ over $\izf + \mathbf{DC}$
using topological models. We obtain here a similar separation result.
\begin{corollary}
  For each $n$ $\llpo_{n + 1}$ does not imply $\llpo_n$ over $\izf +
  \markov + \churchu + \acndn{n} + \hacn$.
\end{corollary}

\begin{proof}
  $\izf + \markov + \churchu + \acndn{n} + \hacn + \llpo_{n + 1}$ is
  consistent by theorem \ref{thm:ctllpocon} and proposition
  \ref{prop:acimpl} but $\izf + \markov + \churchu + \acndn{n} +
  \llpo_{n}$ is not by theorem \ref{thm:richmansthm}.
\end{proof}

In addition we get the following corollary by the same argument.

\begin{corollary}
  $\acndn{n}$ does not imply $\acndn{n + 1}$ over $\izf + \markov +
  \churchu + \llpo_{n + 1} + \hacn$.
\end{corollary}

\begin{proof}
  $\izf + \markov + \churchu + \llpo_{n + 1} + \hacn + \acndn{n}$ is
  consistent by theorem \ref{thm:ctllpocon} and proposition
  \ref{prop:acimpl} but $\izf + \markov + \churchu + \llpo_{n + 1} +
  \acndn{n + 1}$ is not by theorem \ref{thm:richmansthm}.
\end{proof}

\subsection{Existence Properties}

\begin{theorem}
  \label{thm:2}
  Let $T$ be one of $\czf$ or $\izf$. Let $\phi(x)$ be a formula with
  one free variable, $x$. Suppose that
  \begin{equation*}
    \label{eq:5}
    T + \markov + \llpo_n \quad \vdash \quad (\exists j \in \nat)\,\phi(j)
  \end{equation*}
  Then there is a finite set $J \subseteq \nat$ such that
  \begin{equation*}
    T + \markov + \llpo_n \quad \vdash \quad \bigvee_{j \in J} \phi(\underline{j})
  \end{equation*}
\end{theorem}

\begin{proof}
  Suppose that 
  \begin{equation*}
    T + \markov + \llpo_n \quad \vdash \quad (\exists j \in \nat)\,\phi(j)
  \end{equation*}
  Then we have by lemma \ref{lem:llponsound} that
  \begin{equation*}
    T + \markov + \ipnn \quad \vdash \quad \vln \models (\exists
    j \in \nat)\, \phi(j)
  \end{equation*}
  Fix a primitive recursive encoding of finite sets of naturals as
  naturals. Then by lemma \ref{lem:listwitnesses}, working in $T +
  \markov + \ipnn$ we can prove that there exists a natural number
  encoding a finite set $J$ such that $\vln \models (\exists j \in
  \hat{J})\, \phi(j)$. Now applying theorem \ref{thm:ipnep} and
  absoluteness for primitive recursive formulas we have a finite set
  $J \subseteq \nat$ such that
  \begin{equation*}
    \label{eq:6}
    T + \markov + \ipnn \quad \vdash \quad \vln \models \bigvee_{j \in
    J}\,\phi(\hat{\underline{j}})
  \end{equation*}
  
  By lemma \ref{lem:llpoipnn} we have in particular that,
  \begin{equation*}
    T + \markov + \llpo_n \quad \vdash \quad \vln \models \bigvee_{j \in
    J}\,\phi(\hat{\underline{j}})
  \end{equation*}

  Finally we apply lemma \ref{lem:degeneratevln} to get
  \begin{equation*}
    T + \markov + \llpo_n \quad \vdash \quad 
    \bigvee_{j \in J} \phi(\underline{j})
  \end{equation*}
\end{proof}

\begin{theorem}
  \label{thm:nepmain}
  Let $T$ be one of $\czf$ or $\izf$. Let $n,k \in \nat$ and $k < n$,
  and let $\phi(x)$ be a formula with one free variable, $x$. Suppose
  that
  \begin{equation*}
    T + \markov + \llpo_n \vdash (\exlte{k} x)\,\phi(x)
  \end{equation*}
  Then for some $j \in \nat$ we have
  \begin{align}
    T + \markov + \llpo_n &\vdash \neg \neg \phi(\underline{j}) \label{eq:dntruth}\\
    T + \markov + \llpo_{\lceil \frac{n}{k} \rceil} &\vdash
    \phi(\underline{j}) \label{eq:noverktruth}
  \end{align}
\end{theorem}

\begin{proof}
  Suppose that $T + \markov + \llpo_n \vdash (\exlte{k}
  x)\,\phi(x)$. Then we have by lemma \ref{lem:llponsound} that
  \begin{equation*}
    \label{eq:44}
    T + \markov + \ipnn \quad \vdash \quad \vln \models (\exlte{k}
    x)\, \phi(x)
  \end{equation*}
  Hence, 
  applying lemma \ref{lem:buildwitness} with $p_j := \llbracket
  \phi(\hat{j}) \rrbracket$, and writing
  $\operatorname{Good}(T)$ to mean $T$ is a good $\lceil \frac{n}{k}
  \rceil$ tree and
  $\operatorname{VeryGood}(T)$ to mean $T$ is a very good tree,
  \begin{multline}
    \label{eq:48}
    T + \markov + \ipnn \quad \vdash \quad (\exists j \in \nat)(\exists
    T)\,\operatorname{Good}(T) \; \wedge \\ \operatorname{VeryGood}(T)
    \rightarrow \left( \vln \models \phi(\hat{j}) \right)
  \end{multline}

  We now apply lemma \ref{thm:ipnep} to find $j \in \nat$ such that
  \begin{multline}
    T + \markov + \ipnn \quad \vdash \quad (\exists
    T)\,\operatorname{Good}(T) \; \wedge \\ \operatorname{VeryGood}(T)
    \rightarrow \left( \vln \models \phi(\hat{\underline{j}}) \right)
  \end{multline}
  
  By lemma \ref{lem:llpoipnn} we have in particular that,
  \begin{multline}
    T + \markov + \llpo_n \quad \vdash \quad (\exists
    T)\,\operatorname{Good}(T) \; \wedge \\ \operatorname{VeryGood}(T)
    \rightarrow \left( \vln \models \phi(\hat{\underline{j}}) \right)
  \end{multline}
  
  However, we also have by lemma \ref{lem:degeneratevln} that 
  \begin{equation*}
    \label{eq:49}
    T + \markov + \llpo_n \quad \vdash \quad (\forall j \in \nat)\;\left(\vln
      \models \phi(j) \right) \rightarrow \phi(j)
  \end{equation*}
  
  Finally, we deduce \eqref{eq:dntruth} by corollary
  \ref{cor:nvgoodngood} and deduce \eqref{eq:noverktruth} by theorem
  \ref{thm:llpontrees}.
\end{proof}

\begin{corollary}
  \label{thm:6}
  Let $T$ be one of $\czf$ or $\izf$. Let $n,k \in \nat$ and $k < n$,
  and let $\phi_1,\ldots,\phi_k$ be sentences.
  Suppose that
  \begin{equation*}
    T + \markov + \llpo_n \vdash \bigvee_{i = 1}^k \phi_i
  \end{equation*}
  Then for some $1 \leq i \leq k$ we have
  \begin{align*}
    T + \markov + \llpo_n &\vdash \neg \neg \phi_i \\
    T + \markov + \llpo_{\lceil \frac{n}{k} \rceil} &\vdash \phi_i
  \end{align*}
\end{corollary}

\begin{corollary}
  \label{cor:unep}
  Let $T$ be one of $\czf$ or $\izf$. Let $n \in \nat$
  and let $\phi(x)$ be a formula with one free variable, $x$. Suppose
  that
  \begin{equation*}
    T + \markov + \llpo_n \vdash (\exists ! x \in \nat)\,\phi(x)
  \end{equation*}
  Then for some $j \in \nat$ we have
  \begin{align*}
    T + \markov + \llpo_{n} &\vdash
    \phi(\underline{j})
  \end{align*}
\end{corollary}

\begin{proof}
  This is a special case of \eqref{eq:noverktruth} in theorem
  \ref{thm:nepmain} taking $k := 1$.
\end{proof}

By contrast, we see below that none of these theories can have the
full numerical existence property.

\begin{theorem}
  \label{thm:llpononep}
  The numerical existence property does not hold for any consistent,
  recursively axiomatisable extension of $\czf + \llpo_\infty$.
\end{theorem}

\begin{proof}
  Let $T$ be a consistent recursively axiomatisable extension of $\czf
  + \llpo_\infty$. In fact, a similar proof works for any theory $T$
  that interprets enough first order arithmetic to state
  $\llpo_\infty$ and carry out the constructions used in G\"odel's
  incompleteness theorem. However, for convenience we will use
  definitions and notation from set theory.

  Assume that we are given a bijective pairing on $\nat$ with
  primitive recursive pairing and projection functions, which we write
  as $( , )$, $( )_0$ and $( )_1$ respectively, and let $\prf$ be a
  primitive recursive provability predicate.  

  Construct by diagonalisation a formula $\phi(n)$, where $n$ is the
  only free variable and such that
  \begin{multline}
    \label{eq:23}
    T \vdash (\forall n \in \nat)\,(\phi(n) \; \leftrightarrow \;
    ((\forall m \in \nat)\,((m)_0 = n \wedge \prf((m)_1,
    \godelno{\phi(\underline{n})}) \quad \rightarrow \\
    (\exists m' < m)\,
    \prf((m')_1, \godelno{\phi(\underline{(m')_0})}))))
  \end{multline}

  Write $\psi(n, m)$ for the formula
  \begin{multline}
    \label{eq:45}
    \psi(n, m) := ((m)_0 = n \wedge \prf((m)_1,
    \godelno{\phi(\underline{n})}) \quad \rightarrow \\
    (\exists m' < m)\,
    \prf((m')_1, \godelno{\phi(\underline{(m')_0})})))
  \end{multline}

  Now define for each $n \in \nat$, $\alpha_n \in \natinfty$ as
  follows
  \begin{equation*}
    \label{eq:24}
    \alpha_n(l) :=
    \begin{cases}
      1 & \text{for all } m \leq l, \psi(n, m) \\
      0 & \text{otherwise}
    \end{cases}
  \end{equation*}

  So that we can apply $\llpo_\infty$, we first show that for all $n
  \neq n'$ we have $\alpha_n \vee \alpha_{n'} = 1$. For any $l \in
  \nat$, assume for a contradiction that $\alpha_n \vee \alpha_{n'}
  (l) = 0$. Without loss of generality we may assume $l$ is the least
  such number (since $\psi(n, m)$ is primitive recursive and so
  decidable). By the minimality of $l$ we must have either $\neg
  \psi(n, l)$ or $\neg \psi(n', l)$. However, we cannot have both of
  these since this would imply $(l)_0 = n$ and $(l)_0 = n'$. Hence we
  have without loss of generality $(l)_0 = n$ and since
  $\alpha_{n'}(l) = 0$ and $\psi(n', l)$, there must be some $l' < l$
  such that $\neg \psi(n', l')$. In particular we have $\prf((l')_1,
  \godelno{\phi(\underline{(l')_0})})$ but also for all $m < l$, $\neg
  \prf((m)_1, \godelno{\phi(\underline{(m)_0})})$, giving us a
  contradiction. Therefore, $\alpha_n \vee \alpha_{n'} = 1$ as
  required.

  We can now apply $\llpo_\infty$ to show that $T \vdash (\exists n
  \in \nat)\, \alpha_n = 1$. Note that this implies $T \vdash (\exists
  n \in \nat)\, \phi(n)$.

  Now if we assume that the numerical existence property holds for $T$
  then there must be some $n \in \nat$ such that $T \vdash
  \phi(\underline{n})$. So there must be $m$ such that $(m)_1$ codes a
  proof for $\phi(\underline{(m)_0})$ (by taking $(m)_0 = n$). Since
  the provability predicate is decidable, without loss of generality
  we can take $m$ to be the least number such that $(m)_1$ codes a
  proof for $\phi(\underline{(m)_0})$. By the minimality of $m$ we
  have that for all $m' < m$, $\neg \prf((m')_1,
  \phi(\underline{(m')_0}))$. But this is a $\Delta_0$ sentence, so by
  absoluteness for $\Delta_0$ sentences we have
  \begin{equation*}
    \label{eq:46}
    T \vdash \neg (\exists m' < \underline{m})\,\prf((m')_1,
    \phi(\underline{(m')_0}))
  \end{equation*}
  Again by absoluteness of $\Delta_0$ sentences, we also have
  \begin{equation*}
    \label{eq:47}
    T \vdash (\underline{m})_0 = \underline{(m)_0} \,\wedge\,
    \prf((\underline{m})_1, \godelno{\phi(\underline{(m)_0})})
  \end{equation*}
  Hence we have $T \vdash \neg \phi(\underline{(m)_0})$,
  contradicting that $T \vdash \phi(\underline{(m)_0})$ and the
  consistency of $T$. Therefore the numerical existence property must
  fail for $T$.
\end{proof}

\begin{corollary}
  For every $n$, there is a formula with one free variable, $\phi(x)$,
  such that
  $\izf + \llpo_n + \markov \vdash (\exists x \in \nat)\,\phi(x)$ but
  for every formula $\psi(x)$,
  $\izf + \llpo_n + \markov \nvdash (\exists ! x \in \nat)\,\phi(x)
  \wedge \psi(x)$.
\end{corollary}

\begin{proof}
  Let $\phi(x)$ be the formula from the proof of theorem
  \ref{thm:llpononep}. If
  $\izf + \llpo_n + \markov \vdash (\exists ! x \in \nat)\,\phi(x)
  \wedge \psi(x)$ was provable, then by corollary \ref{cor:unep} there
  would be some $j$ such that
  $\izf + \llpo_n + \markov \vdash \phi(\underline{j}) \wedge
  \psi(\underline{j})$. But in particular this gives
  $\izf + \llpo_n + \markov \vdash \phi(\underline{j})$ contradicting
  theorem \ref{thm:llpononep}.
\end{proof}

In \cite{friedmandpnep}, Friedman showed that for every
recursively axiomatisable extension of Heyting arithmetic the
disjunction property implies the numerical existence property. He
further remarks, without proof, that there is a $\Delta_2^0$ extension
that satisfies the disjunction property but not the numerical
existence property. As a corollary of the above results, we obtain a
reasonably natural example of a $\Pi^0_2$ theory with the disjunction
property but not the numerical existence property.

\begin{corollary}
  \label{cor:dpalln}
  Assume classical logic in the meta theory. The theory $T :=
  \bigcap_n \, \izf + \markov + \llpo_n$ (i.e. the set of formulas
  provable in $\izf + \markov + \llpo_n$ for every $n$) has the
  disjunction property.
\end{corollary}

\begin{proof}
  Suppose that $T \vdash \phi \vee \psi$. Then, for each $n$, $\izf +
  \markov + \llpo_{2 n} \vdash \phi \vee \psi$. Hence either $\izf +
  \markov + \llpo_{n} \vdash \phi$ or $\izf + \markov + \llpo_{n}
  \vdash \psi$. Let $X := \{ n \in \nat \;|\; \izf + \markov +
  \llpo_{n} \vdash \phi \}$ and $Y := \{ n \in \nat \;|\; \izf +
  \markov + \llpo_{n} \vdash \phi \}$. $X$ and $Y$ are downwards
  closed subsets of $\nat$ such that $X \cup Y = \nat$. By classical
  logic we therefore have either $X = \nat$ or $Y = \nat$. Without
  loss of generality, say $X = \nat$. Then we have that for every $n$,
  \begin{equation*}
    \label{eq:27}
    \izf + \markov + \llpo_{n} \vdash \phi
  \end{equation*}
  But we have now shown the disjunction property for this theory.
\end{proof}

\begin{theorem}
  The theory $T := \bigcap_n \, \izf + \markov + \llpo_n$ (i.e. the
  set of formulas provable in $\izf + \markov + \llpo_n$ for every
  $n$) does not have the numerical existence property.
\end{theorem}

\begin{proof}
  Note that the statement $(\exists n \in \nat)\,\llpo_n$ can be
  formalised in set theory and holds in each $\izf + \markov +
  \llpo_n$ for each $n$. However, for each $n$, we have seen that
  $\izf + \markov + \llpo_{n + 1}$ does not prove $\llpo_n$, so it is
  not provable in $T$. Hence $T$ proves $(\exists n \in
  \nat)\,\llpo_n$ but does not prove $\llpo_n$ for any $n$, so the
  numerical existence property fails.
\end{proof}

\subsection{Consistency of Brouwerian Continuity Principles}

Recall that the fan theorem and bar induction are defined as below.

\begin{definition}
  \label{def:10}
  Write $2^\ast$ for the set of finite binary sequences. If $\alpha :
  \nat \rightarrow 2$ is an infinite binary sequence, write
  $\bar{\alpha}(n)$ for the finite binary sequence of length $n$
  obtained by restricting $\alpha$.

  A subset $R$ of $2^\ast$ is a \emph{bar} if for every $\alpha : \nat
  \rightarrow 2$, there exists some $n \in \nat$ such that
  $\bar{\alpha(n)} \in R$.

  A bar, $R$, is \emph{uniform} if there exists $n \in \nat$ such that
  for all $\alpha : \nat \rightarrow 2$, there exists $m \leq n$ such
  that $\bar{\alpha}(m) \in R$.

  The \emph{fan theorem}, $\fan$ is the axiom that every bar is
  uniform.

  A subset $R$ of $\nat^\ast$ is a \emph{bar} if for every $\alpha :
  \nat \rightarrow \nat$, there exists some $n \in \nat$ such that
  $\bar{\alpha(n)} \in R$.

  A bar, $R$, is \emph{monotone} if whenever $s \in R$ and $s'$ is a
  finite binary sequence extending $s$, then also $s' \in R$.

  If $s$ and $t$ are finite binary sequences, write $s \ast t$ for the
  concatenation of $s$ and $t$.

  \emph{Monotone bar induction}, $\bim$, is the following axiom. Let
  $Q \subset \nat^\ast$ be such that there is a monotone bar $R$ with
  $R \subseteq Q$ and $Q$ has the property that whenever $s \ast
  \langle n \rangle \in Q$ for all $n$ also $s \in Q$. Then $\langle
  \rangle \in Q$.
\end{definition}

\begin{proposition}[$\czf + \markov$]
  \label{prop:ktwomarkov}
  Let $\vktwo$ be the function realizability model from
  \cite{rathjenbrouwerian}. Then $\markov$ holds in $\vktwo$.
\end{proposition}

\begin{proof}
  This can easily be checked by applying $\markov$ in the background
  and noting that there is a continuous functional that takes as input
  $\alpha : \nat \rightarrow 2$ such that there exists $n$ with
  $\alpha(n) = 1$ and returns the first $n$ such that $\alpha(n) =
  1$.
\end{proof}

\begin{lemma}[$\czf + \markov + \ip{\mathcal{F}_n}{\baire} + \fan$]
  \label{lem:vlnfan}
  \begin{equation*}
    \label{eq:7}
    \vln \models \fan
  \end{equation*}
\end{lemma}

\begin{proof}
  Let $R \in \vln$ be such that the statement that $R$ is a bar holds
  in $\vln$. We first construct a set $R'$ in the background universe
  and check that $R'$ is a bar. Let $R'$ be the set of $\sigma \in
  2^\ast$ such that $\vln \models (\exists \sigma' \in
  R)\,\sigma' \leq \hat{\sigma}$.

  To show that $R'$ is a bar, let $\alpha \in 2^\nat$. Then $\vln
  \models (\exists j \in \nat)\,\bar{\hat{\alpha}}(j) \in R$, since
  $R$ is internally a bar in $\vln$. Hence by lemma
  \ref{lem:listwitnesses}, there is a finite set $J \subseteq N$ such
  that $\vln \models (\exists j \in J)\,\bar{\hat{\alpha}}(j) \in
  R$. Then set $N := \max J$. We clearly have $\bar{\alpha}(N) \in
  R'$, and so $R'$ is a bar.

  We can now apply $\fan$ in the background universe to find $m$ such
  that for every $\alpha \in 2^\nat$ there exists $l \leq m$ such that
  $\bar{\alpha}(l) \in R'$. But we now have $\vln \models (\exists x
  \leq \hat{m})\, \bar{\alpha}(x) \in R$ as required.
\end{proof}

\begin{lemma}[$\czf + \markov + \ipnn + \bim$]
  \label{lem:vlnbim}
  \begin{equation*}
    \label{eq:53}
    \vln \models \bim
  \end{equation*}
\end{lemma}

\begin{proof}
  Suppose that $R, Q \in \vln$ are such that in $\vln$ the following
  holds: $R \subseteq Q \subseteq \nat^\ast$, $R$ is a monotone bar and
  whenever $Q$ contains every immediate successor of $\sigma \in
  \nat^\ast$, it also contains $\sigma$. We first define external
  versions of $R$ and $Q$ as follows:
  \begin{align*}
    \label{eq:54}
    R' &:= \{ \sigma \in \nat^\ast \;|\; \vln \models \hat{\sigma} \in R \}
    \\
    Q' &:= \{ \sigma \in \nat^\ast \;|\; \vln \models \hat{\sigma} \in Q \}
  \end{align*}

  Note that we can easily show $R' \subseteq Q' \subseteq \nat^\ast$
  and that $R'$ is monotone. To apply $\bim$ in the background, it
  only remains to check that $R'$ is a bar and that for any $\sigma
  \in \nat^\ast$ if $Q'$ contains every immediate successor of
  $\sigma$ it also contains $\sigma$.

  To check that $R'$ is a bar, let $f : \nat \rightarrow \nat$. Then
  $\vln \models (\exists x \in \nat)\,\bar{\hat{f}}(x) \in R$. Hence
  by lemma \ref{lem:listwitnesses}, there is a finite set
  $J \subseteq N$ such that
  $\vln \models \bigvee_{j \in J}\bar{\hat{f}}(j) \in R$. Then set
  $N := \max J$. By monotonicity we have that for each $j \in J$,
  $\vln \models \bar{\hat{f}}(j) \in R \rightarrow \bar{\hat{f}}(N)
  \in R$.
  So we deduce that $\vln \models \bar{\hat{f}}(N) \in R$ and so
  $\bar{f}(N) \in R'$. Therefore $R'$ is a bar as required.

  Now let $\sigma \in \nat^\ast$ be such that for all $m \in \nat$,
  $\sigma \ast \langle m \rangle \in Q'$. Then by absoluteness, we
  have $\vln \models (\forall x \in \nat)\,\hat{\sigma} \ast \langle x
  \rangle \in Q$. Therefore, $\vln \models \hat{\sigma} \in Q$ and so
  $\sigma \in Q'$.

  We can now apply $\bim$ in the background to deduce that $\langle
  \rangle \in Q'$. Therefore $\vln \models \langle \rangle \in Q$. So
  we have confirmed $\bim$ holds in $\vln$ as required.
\end{proof}

\begin{lemma}[$\czf + \markov + \ipnn + \contbn$]
  \label{lem:vlncontbn}
  \begin{equation*}
    \vln \models \contbn
  \end{equation*}
\end{lemma}

\begin{proof}
  Suppose $\vln \models F : \baire \rightarrow \nat$. Then by lemma
  \ref{lem:function2abs} there is $G : \baire \rightarrow \nat$ such
  that $\vln \models F = \hat{G}$. Let $\alpha \in \baire$. By
  $\contbn$ in the background, there exists $j$ such that for any
  $\beta \in \baire$, $\bar{\alpha}(j) = \bar{\beta}(j)$ implies
  $G(\alpha) = G(\beta)$. However, by absoluteness we then have $\vln
  \models (\forall \beta \in \baire)\,\bar{\alpha}(j) =
  \bar{\beta}(j) \rightarrow G(\alpha) = G(\beta)$. But we now have
  that in $\vln$, $\hat{G}$ and so also $F$ are continuous. We deduce
  $\contbn$ in $\vln$.
\end{proof}

\begin{lemma}[$\czf + \markov + \ipnn + \choice_2$]
  \label{lem:vlnbc}
  Let $n, k \in \nat$ and $2 \leq k < n$. Then
  \begin{equation*}
    \vln \models \acb{k}{\lceil \frac{n}{k} \rceil}
  \end{equation*}
\end{lemma}

\begin{proof}
  By adapting the proof of lemma \ref{lem:vlncc} and applying $\choice_2$
  in the background.
\end{proof}

\begin{lemma}[$\czf + \markov + \ipnn + \choice_2$]
  \label{lem:vlnhbc}
  \begin{equation*}
    \vln \models \hacb
  \end{equation*}
\end{lemma}

\begin{proof}
  By adapting the proof of lemma \ref{lem:hacnnsound} and applying
  $\choice_2$ in the background.
\end{proof}

\begin{theorem}
  \label{thm:ktwoipcon}
  Assume $\czf$ is consistent. Then for each $n$, so is the following
  theory.
  \begin{equation}
    \label{eq:20}
    \czf + \contac + \fan + \choice_2 + \rdc + \markov +
    \ip{\mathcal{F}_n}{\baire}
  \end{equation}

  Assume $\czf + \markov + \rea$ is consistent. Then for each $n$, so
  is the following theory.
  \begin{equation}
    \label{eq:55}
    \czf + \rea + \contac + \bim + \choice_2 + \rdc + \markov +
    \ip{\mathcal{F}_n}{\baire}
  \end{equation}
\end{theorem}

\begin{proof}
  Using proposition \ref{prop:ktwomarkov} and lemma
  \ref{lem:ipnnktworealized} one can easily adapt the proof of
  \cite[Theorem 9.10]{rathjenbrouwerian} to show this.
\end{proof}

\begin{theorem}
  If $\czf$ is consistent then for each $n$, the following theory is
  also consistent.
  \begin{equation}
    \label{eq:28}
    \czf + \markov + \bigwedge_{2 \leq k < n} \acb{k}{\lceil \frac{n}{k}
      \rceil} + \hacb + \llpo_n + \contbn + \fan
  \end{equation}

  If $\czf + \rea$ is consistent then for each $n$, the
  following theory is also consistent.
  \begin{equation}
    \czf + \markov + \bigwedge_{2 \leq k < n} \acb{k}{\lceil \frac{n}{k}
      \rceil} + \hacb + \llpo_n + \contbn + \bim
  \end{equation}
\end{theorem}

\begin{proof}
  We build $\vln$ in the theory $\eqref{eq:20}$, which is consistent
  by theorem \ref{thm:ktwoipcon}. We then have that $\vln$ models
  \eqref{eq:28} by lemmas \ref{lem:vlnmarkov}, \ref{lem:vlnbc},
  \ref{lem:vlnhbc}, \ref{lem:llponsound}, \ref{lem:vlncontbn} and
  \ref{lem:vlnfan}. To add monotone bar induction we also apply lemma
  \ref{lem:vlnbim}.
\end{proof}

(There is already a similar result for $\llpo$ over second order
arithmetic due to Van Oosten in \cite[Section 5]{vanoostenlifschitz}.)

\begin{corollary}
  $\czf + \markov + \bigwedge_{2 \leq k < n} \acb{k}{\lceil
    \frac{n}{k} \rceil} + \hacb + \llpo_n + \contbn + \bim$ does not prove
  $\lcp$ or $\choice_{\baire,2}$.
\end{corollary}

\begin{proof}
  $\czf + \markov + \bigwedge_{2 \leq k < n} \acb{k}{\lceil
    \frac{n}{k} \rceil} + \hacb + \llpo_n + \contbn + \bim$ is consistent, so
  it suffices to show $\czf + \llpo_n + \lcp$ and $\czf + \llpo_n +
  \contbn + \choice_{\baire,2}$ are not.

  In both cases, we show the theories are inconsistent by first
  noting that there is a surjection $F : \baire \twoheadrightarrow
  \{\langle \alpha_1,\ldots, \alpha_n \rangle \in \natinfty^n \;|\;
  \alpha_i \vee \alpha_j = 1, \text{ for } i \neq j \}$, defined as
  follows. 
  \begin{equation*}
    \label{eq:57}
    (F(\alpha))_i(k) =
    \begin{cases}
      0 & \alpha(k')
      \equiv i \mod n + 1 \text{ where } k' \leq k \text{
        least s.t. } \alpha(k') \neq 0  \\
      1 & \text{otherwise}
    \end{cases}
  \end{equation*}
  
  By $\llpo_n$, there is $1 \leq i \leq n$ for each $\alpha \in
  \baire$ such that $(F(\alpha))_i = 1$. Let $\alpha$ be such that
  $(F(\alpha))_i = 1$ for all $i$. By $\lcp$ there is some $i, k \in
  \nat$ such that whenever $\bar{\beta}(k) = \bar{\alpha}(k)$,
  $(F(\beta))_i = 1$. However, we can now easily find $\beta$ such
  that $\bar{\beta}(k) = \bar{\alpha}(k)$ but $(F(\beta))_i \neq 1$ to
  get a contradiction. Similarly, we can use $\choice_{\baire, 2}$ to
  get a function $G: \baire \rightarrow \nat$ such that for all
  $\alpha$, $(F(\alpha))_{G(\alpha)} = 1$, contradicting $\contbn$.
\end{proof}

\section{Connections to Other Formal Systems}
\label{sec:sugg-furth-work}

\subsection{Connections to Topos Theory}
\label{sec:conn-topos-theory}

\newcommand{\eff}{\mathsf{Eff}}
\newcommand{\set}{\mathsf{Set}}

The $\mathcal{L}_n$ considered in this paper appear to be strongly
related to the local operators in the effective topos previously
considered by Lee and Van Oosten in \cite{leevanoosten}, specifically
to the local operators corresponding to finitary sights. We expect
that in fact these local operators can be obtained by carrying out the
construction of $\mathcal{L}_n$ in the effective topos. The
realizability model $\vktwo$ corresponds to the topos
$\mathsf{RT}(\kltwo)$ (as described, for example, in \cite[Section
4.3]{vanoosten}). Since we only require computable functions, one
might expect our constructions to work also in the relative
realizability topos $\mathsf{RT}(\kltwo^\mathrm{REC}, \kltwo)$ (see
\cite[Section 4.5]{vanoosten}). The realizability with truth model is
related to the topos $(\eff \downarrow \Delta)$ obtained by gluing
along the inclusion functor from $\set$ to $\eff$. Putting this all
together, we make the following conjecture.

\begin{conjecture}
  Some of the local operators in $\eff$ considered in
  \cite{leevanoosten} have counterparts in the toposes
  $\mathsf{RT}(\kltwo)$, $\mathsf{RT}(\kltwo^\mathrm{REC}, \kltwo)$
  and $(\eff \downarrow \Delta)$.
\end{conjecture}

(We again point out that Van Oosten has already shown that the
original Lifschitz realizability model has a counterpart in
$\mathsf{RT}(\kltwo)$ (see \cite[Section 4.3]{vanoosten}) and for
$\mathbf{q}$-realizability (an ancestor of realizability with truth)
(see \cite[Proposition 3.5]{vanoostenlifschitz})).

\subsection{Connections to Type Theory}
\label{sec:conn-type-theory}

\begin{definition}
  Let $\Gamma$ be a context in type theory. We say that $\Gamma$ has
  \emph{propositional canonicity for $\nat$} if whenever $\Gamma
  \vdash t : \nat$, there is some $n \in \nat$ and a term $p$ such
  that $\Gamma \vdash p : \operatorname{Id}_\nat (t, \underline{n})$.
\end{definition}

Suppose we are working in a variant of type theory that has a
propositional truncation operator (such as type theory with brackets,
as in \cite{awodeybauerpat}). In such theories there are two different
ways of formalising $\llpo$ depending on whether or not we use the
propositional truncation operator $\| - \|$. We call these $\llpo_+$
and $\llpo_\vee$ and define them as follows.

\begin{align*}
  \llpo_+ &:= \prod_{\alpha : \nat \rightarrow 2} \left(\prod_{m, n :
      \nat} (\alpha(m) = 1 + \alpha(n) = 1 \;\rightarrow\; m = n)
  \right) \quad \rightarrow \nonumber\\ & \qquad \left( \left(\prod_{n
        : \nat} \alpha (2 n) = 0\right) + \left(\prod_{n : \nat}
      \alpha (2 n + 1) =
      0\right)
    \right) \\
    \llpo_\vee &:= \prod_{\alpha : \nat \rightarrow 2} \left(\prod_{m,
        n : \nat} (\alpha(m) = 1 + \alpha(n) = 1 \;\rightarrow\; m =
      n) \right) \quad \rightarrow \nonumber\\ & \qquad \left\|
      \left(\prod_{n : \nat} \alpha (2 n) = 0\right) + \left(\prod_{n : \nat} \alpha
      (2 n + 1) = 0\right) \right\|
\end{align*}

By adapting the proof of theorem \ref{thm:llpononep}, we have,
\begin{theorem}
  The context $(x : \llpo_+)$ does not have propositional canonicity
  for $\nat$ over any variant of type theory for which it is
  consistent (that is, there is no term of type $\bot$ in context $(x
  : \llpo_+)$) and such that the set of judgements is computably
  enumerable.
\end{theorem}

However, we expect by analogy with the results in this paper that the
following holds.
\begin{conjecture}
  The context $(x : \llpo_\vee)$ has propositional canonicity for
  $\nat$ over type theory with bracket types, as studied by Awodey and
  Bauer in \cite{awodeybauerpat}, or similar systems studied by
  Maietti in \cite{maiettimodular}.
\end{conjecture}

\section*{Acknowledgements}
\label{sec:acknowledgements}

This work was supported by the EPSRC project ``Homotopical inductive
types'' through grant No. EP/K023128/1 and by the Hausdorff Research
Institute for Mathematics in Bonn during the trimester ``Types, Sets,
and Constructions'' 2018.  This publication was made possible through
the support of a grant from the John Templeton Foundation (``A new
dawn of intuitionism: mathematical and philosophical advances,'' ID
60842). The opinions expressed in this publication are those of the
authors and do not necessarily reflect the views of the John Templeton
Foundation.

\bibliographystyle{abbrv}
\bibliography{mybib}{}

\begin{thebibliography}{10}

\bibitem{aczelrathjen}
P.~Aczel and M.~Rathjen.
\newblock Notes on constructive set theory.
\newblock Technical Report~40, Institut Mittag-Leffler, 2001.

\bibitem{aczelrathjenbookdraft}
P.~Aczel and M.~Rathjen.
\newblock Notes on constructive set theory.
\newblock Book draft available at
  \url{http://www1.maths.leeds.ac.uk/~rathjen/book.pdf}, 2010.

\bibitem{abhkarithhierarchy}
Y.~Akama, S.~Berardi, S.~Hayashi, and U.~Kohlenbach.
\newblock An arithmetical hierarchy of the law of excluded middle and related
  principles.
\newblock In {\em Proceedings of the 19th Annual IEEE Symposium on Logic in
  Computer Science, 2004.}, pages 192--201, July 2004.

\bibitem{awodeybauerpat}
S.~Awodey and A.~Bauer.
\newblock Propositions as [types].
\newblock {\em Journal of Logic and Computation}, 14(4):447--471, 2004.

\bibitem{rathjenchen}
R.-M. Chen and M.~Rathjen.
\newblock Lifschitz realizability for intuitionistic {Z}ermelo-{F}raenkel set
  theory.
\newblock {\em Archive for Mathematical Logic}, 51(7-8):789--818, 2012.

\bibitem{friedman73}
H.~Friedman.
\newblock The consistency of classical set theory relative to a set theory with
  intuitionistic logic.
\newblock {\em Journal of Symbolic Logic}, 38:315--319, 1973.

\bibitem{friedmandpnep}
H.~Friedman.
\newblock The disjunction property implies the numerical existence property.
\newblock {\em Proceedings of the National Academy of Sciences of the United
  States of America}, 72(8):2877--2878, 1975.

\bibitem{gambinohvi}
N.~Gambino.
\newblock Heyting-valued interpretations for constructive set theory.
\newblock {\em Annals of Pure and Applied Logic}, 137(1–3):164 -- 188, 2006.
\newblock Papers presented at the 2nd Workshop on Formal Topology (2WFTop
  2002).

\bibitem{hendtlasslubarsky}
M.~Hendtlass and R.~Lubarsky.
\newblock Separating fragments of {WLEM}, {LPO}, and {MP}.
\newblock {\em The Journal of Symbolic Logic}, 81(4):1315--1343, 2016.

\bibitem{leevanoosten}
S.~Lee and J.~van Oosten.
\newblock Basic subtoposes of the effective topos.
\newblock {\em Annals of Pure and Applied Logic}, 164(9):866 -- 883, 2013.

\bibitem{lifschitzrealiz}
V.~Lifschitz.
\newblock {${\rm CT}_{0}$}\ is stronger than {${\rm CT}_{0}!$}.
\newblock {\em Proc. Amer. Math. Soc.}, 73(1):101--106, 1979.

\bibitem{maiettimodular}
M.~E. Maietti.
\newblock Modular correspondence between dependent type theories and categories
  including pretopoi and topoi.
\newblock {\em Mathematical Structures in Computer Science}, 15:1089--1149, 12
  2005.

\bibitem{mccarty}
D.~C. McCarty.
\newblock {\em Realizability and Recursive Mathematics}.
\newblock PhD thesis, Ohio State University, 1984.

\bibitem{mylatzthesis}
U.~Mylatz.
\newblock {\em Vergleich unstetiger Funktionen: Principle of Omniscience und \
  Vollst\"andigkeit in der C-Hierarchie}.
\newblock PhD thesis, Faculty for Mathematics and Computer Science, University
  Hagen, 2006.

\bibitem{rathjenbrouwerian}
M.~Rathjen.
\newblock Constructive set theory and {B}rouwerian principles.
\newblock {\em Journal of Universal Computer Science}, 11(12):2008--2033,
  December 2005.

\bibitem{rathjen05}
M.~Rathjen.
\newblock The disjunction and other properties for {C}onstructive
  {Z}ermelo-{F}rankel set theory.
\newblock {\em Journal of Symbolic Logic}, 70:1233--1254, 2005.

\bibitem{rathjen06}
M.~Rathjen.
\newblock Realizability for constructive {Z}ermelo-{F}raenkel set theory.
\newblock In V.~Stoltenberg-Hansen and J.~V\"{a}\"{a}n\"{a}nen, editors, {\em
  Logic Colloquium '03}. Association for Symbolic Logic, 2006.

\bibitem{rathjenlpo}
M.~Rathjen.
\newblock Constructive {Z}ermelo-{F}raenkel set theory and the limited
  principle of omniscience.
\newblock {\em Annals of Pure and Applied Logic}, 165(2):563 -- 572, 2014.

\bibitem{richmanllpon}
F.~Richman.
\newblock Polynomials and linear transformations.
\newblock {\em Linear Algebra and its Applications}, 131:131--137, 1990.

\bibitem{vdbergherbrandtop}
B.~van~den Berg.
\newblock The {H}erbrand topos.
\newblock {\em Mathematical Proceedings of the Cambridge Philosophical
  Society}, 155(2):361–374, 2013.

\bibitem{vanoostenlifschitz}
J.~van Oosten.
\newblock Lifschitz' realizability.
\newblock {\em The Journal of Symbolic Logic}, 55(2):pp. 805--821, 1990.

\bibitem{vanoostentworemarks}
J.~van Oosten.
\newblock Two remarks on the {L}ifschitz realizability topos.
\newblock {\em The Journal of Symbolic Logic}, 61(1):pp. 70--79, 1996.

\bibitem{vanoosten}
J.~van Oosten.
\newblock {\em Realizability: An Introduction to its Categorical Side}, volume
  152 of {\em Studies in Logic and the Foundations of Mathematics}.
\newblock Elsevier, 2008.

\end{thebibliography}
\end{document}